\documentclass[12pt,a4paper]{article}
\usepackage{latexsym,amssymb,amsfonts,amsmath,amsthm,nccmath,enumitem,setspace,graphicx,float}
\usepackage[dvipsnames]{xcolor}
\usepackage[hidelinks]{hyperref}
\usepackage[margin=2cm]{geometry}

\newtheorem{theorem}{Theorem}[section]
\newtheorem{corollary}[theorem]{Corollary}
\newtheorem{lemma}[theorem]{Lemma}

\theoremstyle{definition}
\newtheorem{definition}[theorem]{Definition}

\newtheorem{example}[theorem]{Example}

\newcommand{\B}{\mathcal{B}}

\newcommand{\A}{\mathcal{A}}

\newcommand{\e}{\epsilon}
\renewcommand{\l}{\lambda}

\def \deg {{\rm deg}}
\def \codeg {{\rm codeg}}

\def \leq {\leqslant}
\def \geq {\geqslant}

\def \mod{\pmod}

\let\oldproofname=\proofname
\renewcommand{\proofname}{\rm\bf{\oldproofname}}

\title{Exact values for some unbalanced Zarankiewicz numbers}

\author{Guangzhou Chen\thanks{Henan Engineering Laboratory for Big Data Statistical Analysis and Optimal Control, School of Mathematics and Information Science, Henan Normal University, Xinxiang, 453007, PR China}\qquad
Daniel Horsley\thanks{School of Mathematics, Monash University, VIC 3800, Australia \quad \texttt{danhorsley@gmail.com}}\qquad
Adam Mammoliti\thanks{School of Mathematics, Monash University, VIC 3800, Australia}}
\date{}

\begin{document}
\maketitle
\setstretch{1.2}
\vspace{-5mm}
\begin{abstract}
For positive integers $s$, $t$, $m$ and $n$, the Zarankiewicz number $Z_{s,t}(m,n)$ is defined to be the maximum number of edges in a bipartite graph with parts of sizes $m$ and $n$ that has no complete bipartite subgraph containing $s$ vertices in the part of size $m$ and $t$ vertices in the part of size $n$. A simple argument shows that, for each $t \geq 2$, $Z_{2,t}(m,n)=(t-1)\binom{m}{2}+n$ when $n \geq (t-1)\binom{m}{2}$. Here, for large $m$, we determine the exact value of $Z_{2,t}(m,n)$ in almost all of the remaining cases where $n=\Theta(tm^2)$. We establish a new family of upper bounds on $Z_{2,t}(m,n)$ which complement a family already obtained by Roman. We then prove that the floor of the best of these bounds is almost always achieved. We also show that there are cases in which this floor cannot be achieved and others in which determining whether it is achieved is likely a very hard problem. Our results are proved by viewing the problem through the lens of linear hypergraphs and our constructions make use of existing results on edge decompositions of dense graphs.
\end{abstract}

\section{Introduction}\label{S:intro}

For positive integers $s$, $t$, $m$ and $n$, the Zarankiewicz problem \cite{Zar} asks for the maximum number of edges in a bipartite graph with parts $X$ and $Y$, where $|X|=m$ and $|Y|=n$, that has no complete bipartite subgraph containing $s$ vertices in $X$ and $t$ vertices in $Y$. This maximum is denoted here by $Z_{s,t}(m,n)$.

A simple argument of \v{C}ul\'{\i}k \cite{Cul} shows that, for each $t \geq 2$, $Z_{2,t}(m,n)=(t-1)\binom{m}{2}+n$ when $n \geq (t-1)\binom{m}{2}$. Our focus here is on determining, for large values of $m$, exact values for $Z_{2,t}(m,n)$ in the remaining cases where $n=\Theta(tm^2)$. We will show that, in the cases that concern us, $Z_{2,t}(m,n)$ usually attains the floor of one of two families of upper bounds. The first of these families is due to Roman \cite{Rom}, while the second is new. For positive integers $\l$, $m$ and $n$ we let
\begin{align*}
  A_\l^k(m,n) &= \tfrac{k+1}{2}n+\tfrac{\l}{k}\tbinom{m}{2} && \hbox{for each integer $k \geq 1$} \\[1mm]
  B_\l^k(m,n) &= \tfrac{\l(m-1)-\rho}{k-1}m+\tfrac{k+1}{k^2-1-\rho}\left(n\tbinom{k}{2}-\l \tbinom{m}{2}+\tfrac{\rho k m}{k+1} \right) && \hbox{for each integer $k \geq 2$}
\end{align*}
where $\rho$ is the least nonnegative integer such that $\l(m-1) \equiv \rho \mod{k-1}$.  The result of \cite{Cul} already mentioned states that $Z_{2,\l+1}(m,n)=A_\l^1(m,n)$ for all $n \geq \l\binom{m}{2}$. A specialisation of a result of Roman \cite{Rom} shows that $Z_{2,\l+1}(m,n) \leq A_\l^k(m,n)$ for all $k \geq 1$.  Here, we show that the functions $B^k_\l(m,n)$ are also upper bounds on $Z_{2,\l+1}(m,n)$.

\begin{theorem}\label{T:upperBounds}
For all positive integers $\l$, $m$ and $n$, we have $Z_{2,\l+1}(m,n) \leq B^{k}_{\l}(m,n)$ for each $k \geq 2$.
\end{theorem}

For cases where $n \geq m$ we state the following corollary which, with the appropriate choice of $k$, gives the best of these bounds in every case.

\begin{corollary}\label{C:upperBounds}
For positive integers $k$, $\l$, $m$ and $n$ with $k \geq 2$ and $m \geq \frac{k(k-1)}{\l}+1$, we have
\[Z_{2,\l+1}(m,n) \leq
\left\{
  \begin{array}{ll}
    B^{k+1}_{\l}(m,n) & \hbox{if $\l\binom{m}{2}/\binom{k+1}{2} \leq n < \left(\l\binom{m}{2}+\beta m\right)/\binom{k+1}{2}$} \\[1.5mm]
    A^k_{\l}(m,n) & \hbox{if $\left(\l\binom{m}{2}+\beta m\right)/\binom{k+1}{2} \leq n \leq \left(\l\binom{m}{2}-\frac{\alpha k }{k+1}m\right)/\binom{k}{2}$} \\[1.5mm]
    B^{k}_{\l}(m,n) & \hbox{if $\left(\l\binom{m}{2}-\frac{\alpha k }{k+1}m\right)/\binom{k}{2} < n \leq \l\binom{m}{2}/\binom{k}{2}$}
  \end{array}
\right.
\]
where $\alpha$ and $\beta$ are the least nonnegative integers such that $\l(m-1)-\alpha \equiv 0 \mod{k-1}$ and $\l(m-1)+\beta \equiv 0 \mod{k}$ respectively.
\end{corollary}

For a given $\l$ and $m$, the function of $n$ given by the least of Roman's bounds $A^k_{\l}(m,n)$ is piecewise linear with a critical point $(\l\binom{m}{2}/\binom{k}{2},\frac{2\l}{k-1}\binom{m}{2})$ for each integer $k \geq 2$. At any (integral) such critical point, it is known \cite{Rei2} that $Z_{2,\l+1}(m,n)$ achieves the bound if and only if there exists an \emph{$(m,k,\l)$-design}: a $k$-uniform hypergraph on $m$ vertices in which any pair of vertices occur together in exactly $\l$ edges. Corollary~\ref{C:upperBounds} shows that, for an interval around any critical value $n=\l\binom{m}{2}/\binom{k}{2}$ for which $\l(m-1) \not\equiv 0 \mod{k-1}$, our new bound $B^k_\l$ improves on this situation. To be more precise, it is an improvement when $\l\binom{m}{2}/\binom{k}{2}-\frac{2\rho m}{(k-1)(k+1)} < n < \l\binom{m}{2}/\binom{k}{2}+\frac{2(k-1-\rho) m}{k(k-1)}$, where $\rho$ is as defined in $B^k_\l$. At the critical value, it can improve the old bound by as much as $\frac{m}{4k+2}$ (this occurs when $\rho=\frac{k-1}{2}$). In the special case where $\l(m-1) \equiv 0 \mod{k-1}$ we have $B^{k}_{\lambda}=A^{k-1}_\l$ and our new bound makes no improvement. Corollary~\ref{C:upperBounds} results in a new piecewise linear function obtained from the previous one by ``truncating'' many of its critical points. See Figure~\ref{F:trunc} for a visualisation of this.

Guy \cite{Guy} proved that $Z_{2,2}(m,n) = \lfloor A^2_1(m,n) \rfloor$ when $\binom{m}{2}-2b \leq n \leq \binom{m}{2}$ where $b$ is the maximum number of edges in a $3$-uniform linear hypergraph on $m$ vertices. It is known that $b=\lfloor\frac{m}{3}\lfloor\frac{m-1}{2}\rfloor\rfloor$ if $m \not\equiv 5 \mod{6}$ and $b=\lfloor\frac{m}{3}\lfloor\frac{m-1}{2}\rfloor\rfloor-1$ if $m \equiv 5 \mod{6}$. For large $m$, we can greatly generalise this result by showing that the floor of the bound given by Corollary~\ref{C:upperBounds} is attained when $n=\Theta(\l m^2)$, except possibly for a limited set of values of $n$.

\begin{theorem}\label{T:boundAchieved}
Let $k \geq 2$ be a fixed integer. There is an $M=M(k)$ such that for each integer $m>M$ and all positive integers $\l$ and $n$, $Z_{2,\l+1}(m,n)$ is equal to
\[
\left\{
  \begin{array}{ll}
    \lfloor B^{k+1}_\l(m,n)\rfloor & \hbox{if $\l\binom{m}{2}/\binom{k+1}{2} \leq n \leq (\l\binom{m}{2}+\beta m)/\binom{k+1}{2}-4(k+2)(k+3)$} \\[1.5mm]
    \lfloor A^k_\l(m,n) \rfloor & \hbox{if $\left(\l\binom{m}{2}+\beta m\right)/\binom{k+1}{2}+4(\frac{k^2}{\l}+1) \leq n \leq \left(\l\binom{m}{2}-\frac{\alpha k}{k+1}m\right)/\binom{k}{2}-4(\frac{k(k-1)}{\l}+1)$} \\[1.5mm]
    \lfloor B^{k}_\l(m,n) \rfloor & \hbox{if $\left(\l\tbinom{m}{2}-\tfrac{\alpha k}{k+1} m\right)/\tbinom{k}{2}+4(k-2)(k-1) \leq n \leq \lambda\binom{m}{2}/\binom{k}{2}$}
  \end{array}
\right.
\]
where $\alpha$ and $\beta$ are the least nonnegative integers such that $\l(m-1)-\alpha \equiv 0 \mod{k-1}$ and $\l(m-1)+\beta \equiv 0 \mod{k}$ respectively.
\end{theorem}

Note that $\l$ is arbitrary in Theorem~\ref{T:boundAchieved}: it need not be fixed as $m$ becomes large. For a given $\l$ and $m$, where $m$ is large, and for $n=\Theta(\l m^2)$, Theorem~\ref{T:boundAchieved} establishes that the floor of the bound given by Corollary~\ref{C:upperBounds} is always attained except possibly for in short intervals (of length at most $2k^2+O(k)$) around the transition points between different bounds. Thus, for any positive constant $\epsilon$ there is a constant $c=c(\epsilon)$ such that, for each pair $(\l,m)$ with $m$ sufficiently large, Theorem~\ref{T:boundAchieved} determines the value of $Z_{2,\l+1}(m,n)$ for all $n \geq \epsilon \l m^2$ except for at most $c$ values of $n$.

\begin{figure}[h]
\includegraphics[width=\textwidth]{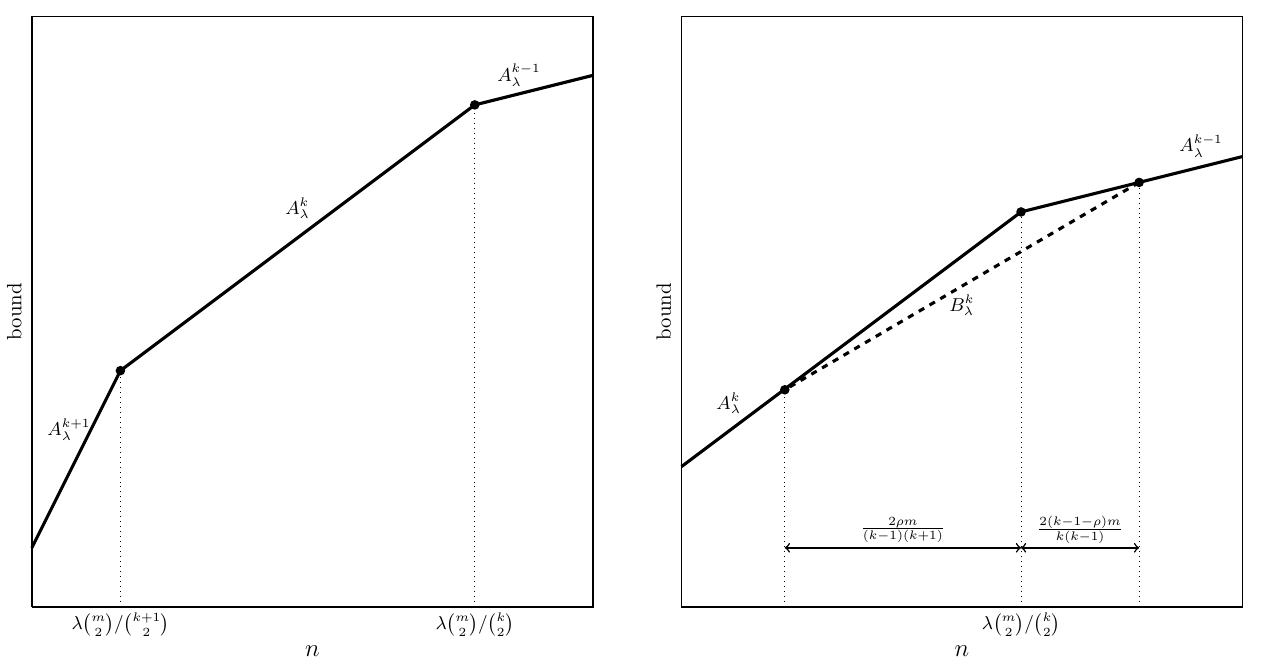}
\caption{On the left is a portion of the piecewise linear function given by the best of Roman's bounds (note that the change in gradients has been exaggerated for clarity). On the right we zoom in around $n=\lambda\binom{m}{2}/\binom{k}{2}$ and see how the bound $B^{k}_{\lambda}$ ``truncates'' the critical point if $\l(m-1) \not\equiv 0 \mod{k-1}$. When $m$ is large compared to $k$, Theorem~\ref{T:boundAchieved} tells us that the floor of these bounds is achieved except possibly in short intervals around the new critical points. When $\l(m-1) \equiv 0 \mod{k-1}$, $B^{k}_{\lambda}=A^{k-1}_\l$ and no improvement is obtained. In such a case Theorem~\ref{T:boundAchieved} tells us that the floor of Roman's bounds is achieved except possibly in a short interval around $n=\lambda\binom{m}{2}/\binom{k}{2}$.}\label{F:trunc}
\end{figure}

One simple consequence of Theorem~\ref{T:boundAchieved} is the following corollary which states that the best of Roman's bounds on $Z_{2,\l+1}(m,n)$ is asymptotically correct when $n=\Theta(\l m^2)$.

\begin{corollary}\label{C:asymptotic}
Let $a>0$ be a fixed constant and let $k$ be the least positive integer such that $a>1/\binom{k+1}{2}$. For $n \sim a \l \binom{m}{2}$ and any positive integer $\lambda$, we have $Z_{2,\l+1}(m,n) \sim \frac{\l}{k}(a\binom{k+1}{2}+1)\binom{m}{2}$ as $m \rightarrow \infty$.
\end{corollary}

The lengths of the intervals for which Theorem~\ref{T:boundAchieved} does not determine $Z_{2,\l+1}(m,n)$ are merely artefacts of our techniques. However we show that, in these intervals, there are cases in which the bound of Corollary~\ref{C:upperBounds} is not attained, and others in which determining whether the bound is attained is likely a hard problem. For positive integers $k \geq 2$ and $\l$, a \emph{symmetric design with block size $k$ and index $\l$} is an $(h,k,\l)$-design where $h= \frac{k(k-1)}{\l}+1$. A symmetric design with block size $k$ and index $1$ is a \emph{projective plane of order $k-1$}.

\begin{theorem}\label{T:boundNotAchieved}
Let $k$, $\l$, $m$ and $n$ be positive integers such that $k \geq 2$ and $\l\binom{m}{2} - n\binom{k}{2} \equiv 0 \mod{k}$. Let $h=\frac{k(k-1)}{\l}+1$.
\begin{itemize}
    \item[\textup{(i)}]
If $\l(m-1) \equiv 0 \mod{k-1}$ and $\l\binom{m}{2}/\binom{k}{2}-\frac{2}{k+1}\max\{h,k+1\} < n < \l\binom{m}{2}/\binom{k}{2}$, then $Z_{2,\l+1}(m,n) < \lfloor A^k_\l(m,n) \rfloor$.
    \item[\textup{(ii)}]
If $\l(m-1) \equiv 0 \mod{k}$ and $\l\binom{m}{2}/\binom{k+1}{2} < n < \l\binom{m}{2}/\binom{k+1}{2} + \frac{2}{k+1}\max\{h,k\}$, then $Z_{2,\l+1}(m,n) < \lfloor A^k_\l(m,n) \rfloor$.
    \item[\textup{(iii)}]
If $\l(m-1) \equiv 0 \mod{k}$ and $n = \l\binom{m}{2}/\binom{k+1}{2}+\frac{2h}{k+1}$, then $Z_{2,\l+1}(m,n)= \lfloor A^k_\l(m,n) \rfloor$ only if there exists a symmetric design with block size $k$ and index $\l$. Further, there is an $M=M(k)$ such that if $m>M$ the `only if' can be replaced with `if and only if'.
\end{itemize}
\end{theorem}

If $h=\frac{k(k-1)}{\l}+1$ is an integer and $m \equiv h \mod{k(k+1)}$, then $\l\binom{m}{2}/\binom{k+1}{2}+\frac{2h}{k+1}$ is an integer and Theorem~\ref{T:boundNotAchieved}(iii) applies. So for each pair $(k,\l)$ such that $h$ is an integer, there are infinitely many pairs $(m,n)$ such that determining whether $Z_{2,\l+1}(m,n)= \lfloor A^k_\l(m,n) \rfloor$ is equivalent to determining whether there exists a symmetric design with block size $k$ and index $\l$. In particular, for each $k \geq 2$, there are infinitely many pairs $(m,n)$ such that determining whether $Z_{2,2}(m,n)= \lfloor A^k_1(m,n) \rfloor$ is equivalent to determining whether there exists a projective plane of order $k-1$. It is well known that, for each $k \geq 2$, $Z_{2,2}(k^2-k+1,k^2-k+1)=k(k^2-k+1)$ if and only if there exists a projective plane of order $k-1$, but it is interesting to see that similar behaviour arises even in the case we consider here of $n=\Theta(\l m^2)$.

There is an extensive history of work on the Zarankiewicz problem, so we only touch here on the results that are most pertinent to our work. Famously, the K\"{o}vari-S\'{o}s-Tur\'{a}n theorem states that $Z_{s,t}(m,n)<(t-1)^{1/s}m n^{1-1/s}+(s-1)n$. This was proved in the case when $s=t$ and $m=n$ in \cite{KovSosTur} but it was observed in \cite{Hyl} that a similar proof gives the more general result. It was also shown in \cite{KovSosTur} that $Z_{2,2}(n,n) \sim n^{3/2}$ as $n \rightarrow \infty$.  Hylt\'{e}n-Cavallius \cite{Hyl} showed that $Z_{2,\lambda+1}(m,n) \leq \frac{1}{2}n+\frac{1}{2}(n^2+4\lambda nm(m-1))^{1/2}$. Reiman \cite{Rei,Rei2} also proved this and observed that equality holds if and only if an $(m,k,\lambda)$-design exists where $k$ is an integer such that $\binom{k}{2}=\lambda\binom{m}{2}/n$. Note that this bound agrees with Roman's bounds at the critical values $n=\l\binom{m}{2}/\binom{k}{2}$ but is inferior to them elsewhere. M\"{o}rs \cite{Mor} showed that, for any fixed $t \geq 2$, $Z_{2,t}(n,n) \sim \sqrt{t-1}n^{3/2}$ as $n \rightarrow \infty$. Alon, R\'{o}nyai and Szab\'{o} \cite{AloRonSza} have proved that, for fixed $s \geq 2$ and $t > s!$, $Z_{s,t}(m,n)=\Theta(mn^{1-1/s})$ if $m^{s/(s+1)} \leq n \leq m^s$ as $m \rightarrow \infty$. Alon, Mellinger, Mubayi and Verstra\"{e}te \cite{AloMelMubVer} showed that $Z_{2,t}(m,n) \sim m\sqrt{n}$ for $t \geq 4$ and $n \sim m^{2/t}$ as $n \rightarrow \infty$. Recently, Conlon \cite{Con} has shown that, for fixed $s$ and $t$ with $2 \leq s \leq t$ and any $m \leq n^{t^{1/(s-1)}/s(s-1)}$, $Z_{s,t}(m,n)=\Omega(mn^{1-1/s})$.

As well as the results of \v{C}ul\'{\i}k \cite{Cul}, Guy \cite{Guy} and Roman \cite{Rom} already discussed, there has been some more recent work that determines exact values for Zarankiewicz numbers. Goddard, Henning and Oellermann \cite{GodHenOel} and Collins, Riasanovsky, Wallace and Radziszowski \cite{ColRiaWalRad} determined exact values for various small parameters. Dam\'{a}sdi, H\'{e}ger and Sz\H{o}nyi \cite{DamHegSzo} found a number of interesting exact values and bounds on $Z_{2,2}(m,n)$ relating to finite projective planes and other designs.

The rest of the paper is arranged as follows. In Section~\ref{S:upperBounds} we establish the new upper bounds and prove Theorem~\ref{T:upperBounds} and Corollary~\ref{C:upperBounds}. Section~\ref{S:prelimsForLBs} details some preliminary results we require for our proof of Theorem~\ref{T:boundAchieved} which then takes place in Sections~\ref{S:achievingA} and~\ref{S:achievingB}. Section~\ref{S:achievingA} is devoted to the cases where $\lfloor A^k_\l(m,n) \rfloor$ is achieved, while Section~\ref{S:achievingB} is devoted to the cases where $\lfloor B^{k+1}_\l(m,n)\rfloor$ or $\lfloor B^{k}_\l(m,n)\rfloor$ is achieved and concludes by completing the proofs of Theorem~\ref{T:boundAchieved} and Corollary~\ref{C:asymptotic}. Theorem~\ref{T:boundNotAchieved} is proved in Section~\ref{S:failing}. We conclude with some discussion of our work and possible future directions in Section~\ref{S:conc}.

\section{Preliminaries and upper bounds}\label{S:upperBounds}

For our purposes, a hypergraph $H$ consists of a set $V(H)$ of vertices and a multiset $E(H)$ of edges, where each edge is a nonempty subset of $V(H)$. We say $H$ is \emph{simple} if $E(H)$ is a set and that $H$ is \emph{$k$-uniform} if every element of $E(H)$ has size $k$. For a vertex $x$ of $H$, $\deg_H(x)$ denotes the number of edges of $H$ that contain $x$ and we define the \emph{total degree} of $H$ to be $\sum_{x \in V(H)}\deg_H(x)$. We also adopt the convention that $\deg_H(x)=0$ for any vertex not in $V(H)$. The minimum and maximum degree of $H$ are denoted by $\delta(H)$ and $\Delta(H)$ respectively. For distinct vertices $x$ and $y$ of $H$, $\codeg_H(x,y)$ denotes the number of edges of $H$ that contain both $x$ and $y$. For a positive integer $\l$, we say $H$ is \emph{$\l$-linear} if $\codeg_H(x,y) \leq \l$ for all distinct $x,y \in V(H)$. The usual definition of linearity is recovered by setting $\l=1$. A \emph{multigraph} is a $2$-uniform hypergraph and a \emph{graph} is a simple multigraph. For a set of $m$ vertices $V$ and a positive integer $\l$, the multigraph $G$ with vertex set $V$ such that $\codeg_{G}(x,y)=\l$ for all distinct $x,y \in V$ is denoted by $\l K_V$ or, if we merely wish to specify its order, by $\l K_{m}$. We omit the $\lambda$ when $\l=1$.

Given a hypergraph $H$ with $m$ vertices and $n$ edges, we can form its incidence graph: a bipartite graph with parts $X$ and $Y$ of sizes $m$ and $n$ whose vertices correspond to the vertices and edges of $H$ and in which two vertices are adjacent if and only if they correspond to an edge of $H$ and a vertex of $H$ in that edge. Similarly, given any bipartite graph with parts $X$ and $Y$ of sizes $m$ and $n$, we can form a corresponding hypergraph with $m$ vertices and $n$ edges. In each case, the hypergraph is $\l$-linear if and only if the bipartite graph contains no complete bipartite subgraph with 2 vertices in $X$ and $\l+1$ vertices in $Y$. Thus we have a well known alternative definition of $Z_{2,\l+1}(m,n)$ as the maximum total degree of a $\l$-linear hypergraph with $m$ vertices and $n$ edges. This is the lens through which we will view the problem for the remainder of the paper.

\begin{definition}
We define $\alpha$ to be the element of $\{0,\ldots,k-2\}$ such that $\l(m-1)-\alpha \equiv 0 \mod{k-1}$ and $\beta$ to be the element of $\{0,\ldots,k-1\}$ such that $\l(m-1)+\beta \equiv 0 \mod{k}$, where the values of $\l$, $m$ and $k$ will always be clear from context.
\end{definition}

We will use $\alpha$ and $\beta$ as given by this definition throughout the rest of the paper. Note that this means that $\rho=\alpha$ in our definition of $B_\l^k(m,n)$, a fact we will tacitly take for granted. (We used $\rho$ rather than $\alpha$ in the definition of $B_\l^k(m,n)$ to avoid confusion in Theorems~\ref{T:upperBounds} and \ref{T:boundAchieved} between the value of $\rho$ in $B_\l^{k+1}(m,n)$ and the value of $\alpha$ in those theorems.)

The bounds $A_\l^k(m,n)$ are the optimal values of a linear program whose variables represent the numbers of edges of each possible size in a $\l$-linear hypergraph of order $m$, whose objective function gives the total degree, and whose constraints specify that the total number of edges must be $n$ and the total codegree must be at most $\l \binom{m}{2}$. By adding to this linear program an additional constraint reflecting the fact that the total codegree at each vertex of the hypergraph must be at most $\l(m-1)$, we obtain a family of linear programs whose optimal values are the bounds $B_\l^k(m,n)$. We formalise these linear programs as follows.

\begin{definition}
For positive integers $k$, $\l$, $m$ and $n$ with $k \geq 2$, define  $\B_\l^k(m,n)$ to be the linear program on nonnegative real variables $n_1,\ldots,n_m$ that aims to maximise $\sum_{i=1}^min_i$ subject to the following constraints.
\begin{align}
\medop\sum_{i=1}^m n_i &= n \label{E:con1}\\
\medop\sum_{i=1}^m\tbinom{i}{2}n_i &\leq \l\tbinom{m}{2} \label{E:con2}\\
\tfrac{1}{k-1-\alpha}\medop\sum_{i=1}^{k-1} i(i-1-\alpha)n_i + \medop\sum_{i=k}^{m}in_i & \leq \tfrac{1}{k-1}m(\l(m-1)-\alpha) \label{E:con3}
\end{align}
Further, define $\A_\l(m,n)$ to be the linear program on nonnegative real variables $n_1,\ldots,n_m$ that aims to maximise $\sum_{i=1}^min_i$ subject to \eqref{E:con1} and \eqref{E:con2}.
\end{definition}

In our next lemma we will show formally that the optimal values of the linear programs we just defined are upper bounds on $Z_{2,\l+1}(m,n)$, but we first give an example that illustrates the intuition behind the constraint \eqref{E:con3}.

\begin{example}\label{EX:deficiency}
Let $\l=1$, $m=99$ and $n=235$. The optimal value of $\A_1(99,235)$ is $A^6_1(99,235)=1631$, achieved when $n_6=14$, $n_7=221$ and $n_i=0$ for each $i \in \{1,\ldots,99\}\setminus \{6,7\}$ (note that this optimum will not always be integer valued). However there cannot exist a linear hypergraph of order 99 with total degree 1631. One way to see this is as follows. Suppose $H$ is a linear hypergraph of order 99 with $n_i$ edges of size $i$ for each $i \in \{1,\ldots,n\}$. A vertex of degree 17 or greater in $H$ must be incident to some edges of size at most 6, for being incident with 17 edges of size at least 7 would result in a vertex with total codegree at least 102, contradicting the fact that the total codegree at each vertex is at most 98. In fact, if we define the \emph{deficiency} of an edge $E$ of $H$ to be $\max\{7-|E|,0\}$ and the deficiency of a vertex to be the sum of the deficiencies of the edges incident with it, then it can be seen that a vertex of degree 17 must have deficiency at least 4 and, in general, that a vertex of degree $17+i$ must have deficiency at least $6i+4$ for any nonnegative integer $i$. From this we can conclude that the total degree $\sum_{i=1}^{99} in_i$ of $H$ is bounded above by $99\times16+\frac{t}{4}$, where $t=\sum_{i=1}^6 i(7-i)n_i$ is the total deficiency contributed by the edges in $H$ (note that an edge contributes its deficiency to each vertex it is incident with). Suitably rearranged, this condition is exactly what \eqref{E:con3} asserts for $k=7$ (and hence $\alpha=2$). We shall see below that, when we take this restriction into account, it emerges that $H$ can have total degree at most $\lfloor B^7_1(99,235) \rfloor = 1628$, where $B^7_1(99,235)$ is the optimal value of $\B^7_1(99,235)$.
\end{example}

\begin{lemma}\label{L:obeyLP}
If there is a $\l$-linear hypergraph $H$ of order $m$ with exactly $n_i$ edges of size $i$ for each $i \in \{1,\ldots,m\}$, then $(n_1,\ldots,n_m)$ satisfies \eqref{E:con2} and, for each integer $k \geq 2$, \eqref{E:con3}.
\end{lemma}

\begin{proof}
The fact that $H$ is $\l$-linear immediately implies that \eqref{E:con2} holds. Let $k \geq 2$ be a fixed integer. We complete the proof by showing that \eqref{E:con3} holds. Let $V=V(H)$ and, similarly to Example~\ref{EX:deficiency}, let the \emph{deficiency} of an edge $E$ of $H$ be $\max\{k-|E|,0\}$. Let $x$ be an arbitrary vertex of $H$. For each $i\in\{1,2,\ldots,m\}$, let $n_i(x)$ be the number of edges of $H$ of size $i$ that contain $x$. Let $n(x)=\sum_{i=1}^mn_i(x)$ be the total number of edges containing $x$ and let $t(x)=\sum_{i=1}^{k-1}(k-i)n_i(x)$ be the total deficiency of the edges incident with $x$. Let $c=\frac{1}{k-1}(\l(m-1)-\alpha)$ be the maximum degree that a vertex of $H$ can have if it is incident with no edges of positive deficiency.

Using the definitions of $n(x)$ and $t(x)$ and the fact $\codeg_H(x,y) \leq \l$ for each $y \in V \setminus \{x\}$, we have that
\begin{equation}\label{E:degBound}
(k-1)n(x) - t(x) =  \medop\sum_{i=1}^{k-1} (i-1)n_i(x)+(k-1)\medop\sum_{i=k}^{m}n_i(x) \leq \medop\sum_{i=1}^{m} (i-1)n_i(x) \leq \l(m-1).
\end{equation}
We claim that $n(x) \leq c+\frac{t(x)}{k-1-\alpha}$. This is obvious if $n(x) \leq c$. Otherwise $n(x) \geq c+1$, and we have
\[t(x) \geq (k-1)(n(x)-c)-\alpha \geq (k-1-\alpha)(n(x)-c)\]
where the first inequality follows from \eqref{E:degBound} after rearranging and substituting $\l(m-1)=c(k-1)+\alpha$, and the second follows because $\alpha \geq 0$ and $n(x)-c \geq 1$. So we do indeed have $n(x) \leq c+\frac{t(x)}{k-1-\alpha}$. Thus,
\begin{equation}\label{E:zBoundByp}
\medop\sum_{i=1}^min_{i} = \medop\sum_{x \in V} n(x) \leq \medop\sum_{x \in V} \left(c+\tfrac{t(x)}{k-1-\alpha}\right) = m c + \tfrac{1}{k-1-\alpha}\medop\sum_{i=1}^{k-1} i(k-i) n_i
\end{equation}
where the last equality follows by applying the definition of $t(x)$ and noting that $\sum_{x \in V}n_i(x)=in_i$ for each $i \in \{1,\ldots,k-1\}$. Using the definition of $c$ and rearranging shows that \eqref{E:zBoundByp} is equivalent to \eqref{E:con3}.
\end{proof}

In view of Lemma~\ref{L:obeyLP}, our next lemma shows that $A_\l^k(m,n)$ is an upper bound on $Z_{2,\l+1}(m,n)$. As mentioned in the introduction, this was established by Roman \cite[Theorem 1]{Rom}. We briefly prove it here, however, in order to demonstrate its relation to the new bounds $B_\l^k(m,n)$.

\begin{lemma}\label{L:aBoundWithEq}
Let $k$, $\l$, $m$ and $n$ be positive integers. The objective value of $\A_\l(m,n)$ is at most $A_\l^k(m,n)$ and this value is achieved if and only if
$\l \binom{m}{2}/\binom{k+1}{2} \leq n \leq \l \binom{m}{2}/\binom{k}{2}$, $n_k=\tfrac{k+1}{2}n-\tfrac{\l}{k}\tbinom{m}{2}$, $n_{k+1}=\tfrac{\l}{k}\tbinom{m}{2}-\tfrac{k-1}{2}n$ and $n_i=0$ for each $i \in \{1,\ldots,m\}\setminus\{k,k+1\}$.
\end{lemma}

\begin{proof}
By adding $\frac{k+1}{2}$ times \eqref{E:con1} to $\frac{1}{k}$ times \eqref{E:con2}, we obtain
\begin{equation}\label{E:weightedSumEasy}
\sum_{i=1}^{m} \left(\tfrac{k+1}{2}+\tfrac{1}{k}\tbinom{i}{2}\right)n_i \leq A_\l^k(m,n).
\end{equation}
Furthermore, we have equality in \eqref{E:weightedSumEasy} if and only if we have equality in \eqref{E:con2}.

Now $\tfrac{k+1}{2}+\tfrac{1}{k}\tbinom{i}{2}-i$ is a convex quadratic in $i$ with zeroes at $i=k$ and $i=k+1$. So $\tfrac{k+1}{2}+\tfrac{1}{k}\tbinom{i}{2}=i$ for $i \in \{k,k+1\}$ and $\tfrac{k+1}{2}+\tfrac{1}{k}\tbinom{i}{2}>i$ for $i \in \{1,\ldots,m\} \setminus \{k,k+1\}$. So we see that the left hand side of \eqref{E:weightedSumEasy} is at least $\sum_{i=1}^min_i$, the objective function of $\A_\l(m,n)$, with equality if and only if $n_i=0$ for each $i \in \{1,\ldots,m\} \setminus\{k,k+1\}$. Thus the objective value of $\A_\l(m,n)$ is at most $A_\l^k(m,n)$ and, for equality to hold, we must have $n_k+n_{k+1}=n$ from \eqref{E:con1} and $\binom{k}{2}n_k+\binom{k+1}{2}n_{k+1}=\l\binom{m}{2}$ from equality in \eqref{E:con2}. Solving these equations and noting that the resulting values of $n_k$ and $n_{k+1}$ must be nonnegative yields the result.
\end{proof}

We can similarly show that $B_\l^k(m,n)$ is an upper bound on $Z_{2,\l+1}(m,n)$ and characterise the cases where we have equality. Recall that when $\l(m-1) \equiv 0 \mod{k-1}$ we have $B_\l^k(m,n)=A_\l^{k-1}(m,n)$ for each $k \geq 2$, so we restrict our attention to cases where $\l(m-1) \not\equiv 0 \mod{k-1}$.

\begin{lemma}\label{L:bBoundWithEq}
Let $k$, $\l$, $m$ and $n$ be positive integers with $k \geq 2$ and $\l(m-1) \not\equiv 0 \mod{k-1}$. The objective value of $\B^k_\l(m,n)$ is at most $B^k_\l(m,n)$ and this value is achieved if and only if $\l \binom{m}{2}-\frac{\alpha k }{k+1}m \leq n\binom{k}{2} \leq \l \binom{m}{2}+m(k-1-\alpha)$ and $n_i=n^*_i$ for each $i \in \{1,\ldots,m\}$ where
\[n^*_i =
\left\{
  \begin{array}{ll}
    \tfrac{(k+1)(k-1-\alpha)}{(k-1)(k^2-1-\alpha)}\bigl(n\tbinom{k}{2}-\l\tbinom{m}{2}+\tfrac{\alpha k}{k+1} m\bigr)\quad & \hbox{if $i = k-1$} \\[1mm]
    \tfrac{\alpha}{k^2-1-\alpha}\bigl(\l\tbinom{m}{2}-n\tbinom{k}{2}+m(k-1-\alpha)\bigr) & \hbox{if $i = k+1$} \\[1mm]
    n-n^*_{k-1}-n^*_{k+1} & \hbox{if $i = k$} \\[1mm]
    0 & \hbox{otherwise.}
  \end{array}
\right.
\]
\end{lemma}

\begin{proof}
Note that $\alpha \in \{1,\ldots,k-2\}$ since $\l(m-1) \not\equiv 0 \mod{k-1}$. By adding $\frac{1}{2}k(k^2-1)$ times \eqref{E:con1}, $k-1$ times \eqref{E:con2} and $k-1-\alpha$ times \eqref{E:con3}, then dividing the whole by $k^2-1-\alpha$, and finally performing some routine simplification, we obtain
\begin{equation}\label{E:weightedSum}
\sum_{i=1}^{m} \mfrac{k(k^2-1)+(k-1)i(i-1)+2i(\min\{k,i\}-1-\alpha)}{2(k^2-1-\alpha)}n_i \leq B_\l^k(m,n)
\end{equation}
where the coefficients of the $n_i$ can be verified by considering the cases $i \in \{1,\ldots,k-1\}$ and $i \in \{k,\ldots,m\}$ separately. Furthermore, we have equality in \eqref{E:weightedSum} if and only if we have equality in both \eqref{E:con2} and \eqref{E:con3}. For brevity, let $c_i$ be the coefficient of $n_i$ on the left hand side of \eqref{E:weightedSum} for each $i \in \{1,\ldots,m\}$. More calculation reveals that
\[c_i-i=
\left\{
  \begin{array}{ll}
    \mfrac{(k+1)(k-i)(k-1-i)}{2(k^2-1-\alpha)}\quad & \hbox{if $i \in \{1,\ldots,k-1\}$} \\[2mm]
    \mfrac{(k-1)(i-k)(i-1-k)}{2(k^2-1-\alpha)} & \hbox{if $i \in \{k,\ldots,m\}$.}
  \end{array}
\right.
\]
From this it is clear that $c_i=i$ if $i \in \{k-1,k,k+1\}$ and $c_i > i$ otherwise. So we see that the left hand side of \eqref{E:weightedSum} is at least $\sum_{i=1}^min_i$, the objective function of $\B_\l^k(m,n)$, with equality if and only if $n_i=0$ for each $i \in \{1,\ldots,m\} \setminus \{k-1,k,k+1\}$. So the objective value of $\B_\l^k(m,n)$ is at most $B_\l^k(m,n)$ with equality if and only if equality holds in both \eqref{E:con2} and \eqref{E:con3} and $n_i=0$ for each $i \notin \{k-1,k,k+1\}$. When \eqref{E:con1} holds and $n_i=0$ for each $i \notin \{k-1,k,k+1\}$, equality in \eqref{E:con2} and \eqref{E:con3} reduces to the following.
\begin{align}
  n\tbinom{k}{2}-(k-1)n_{k-1}+kn_{k+1} &= \l\tbinom{m}{2} \label{E:con2var}\\
  nk-\tfrac{2k-2-\alpha}{k-1-\alpha}n_{k-1}+n_{k+1} &= \tfrac{1}{k-1}m(\l(m-1)-\alpha)\label{E:con3var}
\end{align}
Thus, because $n_{k-1}=n^*_{k-1}$ and $n_{k+1}=n^*_{k+1}$ are the unique solutions to this system of equations, the objective function of $\B^k_\l(m,n)$ equals $B^k_\l(m,n)$ if and only if $(n_1,\ldots,n_m)=(n^*_1,\ldots,n^*_m)$. Observing that $n^*_{k-1}$ and $n^*_{k+1}$ must be nonnegative gives $\l\binom{m}{2}-\frac{\alpha k }{k+1}m \leq n\binom{k}{2} \leq \l\binom{m}{2}+m(k-1-\alpha)$.
\end{proof}

In the above lemma, note that the condition $\l \binom{m}{2}-\frac{\alpha k }{k+1}m \leq n\binom{k}{2} \leq \l \binom{m}{2}+m(k-1-\alpha)$ does not guarantee that the value of $n_k^*$ will be nonnegative and hence that the given solution will be feasible. However, the value of $n_k^*$ will be nonnegative whenever $m \geq \frac{k(k-1)}{\l}+1$. To confirm this, it is sufficient to check it for $n$ at each end of the prescribed interval, because the value of $n_k^*$ is a linear function in $n$.

In general, the optima given by Lemma~\ref{L:bBoundWithEq} have three nonzero variables, in contrast to the optima given by Lemma~\ref{L:aBoundWithEq} which have only two. This means that while hypergraphs witnessing that $Z_{2,\l+1}(m,n)=A^k_\l(m,n)$ might only have edges of two different sizes, hypergraphs witnessing that $Z_{2,\l+1}(m,n)=B^k_\l(m,n)$ must have edges of three different sizes. We will see the effects of this in Sections~\ref{S:achievingA} and \ref{S:achievingB}. We now prove Theorem~\ref{T:upperBounds} and then discuss Corollary~\ref{C:upperBounds}.

\begin{proof}[\textup{\textbf{Proof of Theorem~\ref{T:upperBounds}.}}]
Lemmas~\ref{L:obeyLP}, \ref{L:aBoundWithEq} and \ref{L:bBoundWithEq} establish that $Z_{2,\l+1}(m,n) \leq A_\l^k(m,n)$ and that, for each integer $k \geq 2$, if $\l(m-1) \not\equiv 0 \mod{k-1}$ then $Z_{2,\l+1}(m,n) \leq B_\l^k(m,n)$. Furthermore we have seen that, for each $k \geq 2$, we have $B_\l^k(m,n)=A_\l^{k-1}(m,n)$ when $\l(m-1) \equiv 0 \mod{k-1}$. The result follows.
\end{proof}

Although Corollary~\ref{C:upperBounds} follows immediately from Theorem~\ref{T:upperBounds} and \cite{Rom}, we briefly discuss some aspects of it. Firstly, note that the lower interval of $n$ values given in the statement of Corollary~\ref{C:upperBounds} is empty if $\l(m-1) \equiv 0 \mod{k}$ and the upper interval of $n$ values is empty if $\l(m-1) \equiv 0 \mod{k-1}$.

Next we note that the upper and lower intervals of $n$ values are disjoint. To see this, first note that $\l(m-1) \equiv \alpha k+\beta(k-1) \mod{k(k-1)}$ by the Chinese remainder theorem because $\l(m-1) \equiv \alpha \mod{k-1}$ and $\l(m-1) \equiv -\beta \mod{k}$. Thus we have
\[\left(\l\tbinom{m}{2}-\tfrac{\alpha k }{k+1}m\right)/\tbinom{k}{2}-\left(\l\tbinom{m}{2}+\beta m\right)/\tbinom{k+1}{2} = \mfrac{2m(\l(m-1)-\alpha k-\beta(k-1))}{k(k^2-1)} \geq 0\]
where the last inequality can be seen to hold using the facts $\alpha < k-1$, $\beta < k$, $\l(m-1) \geq k(k-1)$ and $\l(m-1) \equiv \alpha k+\beta(k-1) \mod{k(k-1)}$.

Finally we indicate how to establish that Corollary~\ref{C:upperBounds} does indeed state the best of the bounds in each case.
Let $\ell$, $\l$ and $m$ be positive integers with $\ell \geq 2$ and let $\alpha_\ell$ and $\beta_\ell$ be the least nonnegative integers such that $\l(m-1)-\alpha_\ell \equiv 0 \mod{\ell-1}$ and $\l(m-1)+\beta_\ell \equiv 0 \mod{\ell}$ respectively. It is routine to calculate that $A^{\ell}_\l(m,n) \leq B^{\ell}_\l(m,n)$ when $n \leq (\l\binom{m}{2}-\frac{\alpha_\ell \ell m}{\ell+1})/\binom{\ell}{2}$ and $A^{\ell}_\l(m,n) \geq B^{\ell}_\l(m,n)$ otherwise. Similarly $B^{\ell+1}_\l(m,n) \leq A^{\ell}_\l(m,n)$ when $n \leq (\l\binom{m}{2}+\beta_\ell m)/\binom{\ell+1}{2}$ and $B^{\ell+1}_\l(m,n) \geq A^{\ell}_\l(m,n)$ otherwise. Using these facts inductively for the integers $\ell \geq 2$ shows that the bounds stated in Corollary~\ref{C:upperBounds} are in each case the best of all the bounds $A^{\ell}_\l(m,n)$ and $B^{\ell}_\l(m,n)$. Note that, for any $\ell>k$, the bound $B^{\ell+1}_\l(m,n)$ is not competitive in the cases we are interested in because
\[\left(\l \tbinom{m}{2}+\beta_\ell m\right)/\tbinom{\ell+1}{2} \leq \l \tbinom{m}{2}/\tbinom{k+1}{2}.\]
This inequality can be confirmed using $\beta_\ell \leq \ell-1$ and  $m \geq \frac{k(k-1)}{\l}+1$ except for possibly when $(\ell,\beta_\ell)=(k+1,k)$ or $(\ell,\beta_\ell,k)=(4,3,2)$ . When $(\ell,\beta_\ell)=(k+1,k)$, we must have $\l(m-1) \equiv -k \mod{k+1}$ and hence in fact $\l(m-1) \geq k^2$, and this can be used to confirm the inequality. When $(\ell,\beta_\ell,k)=(4,3,2)$, we must have $\l(m-1) \equiv 1 \mod{4}$ and hence $\l(m-1) \geq 5$ and again we can confirm the inequality.

\section{Preliminaries for Theorems~\ref{T:boundAchieved} and \ref{T:boundNotAchieved}}\label{S:prelimsForLBs}

We first introduce some more notation relating to hypergraphs. Let $\lambda$ be a positive integer. For a vertex $x$ of a hypergraph $H$, we define $N_H^\l(x)$, the \emph{$\l$-neigbourhood of $x$}, to be the set $\{y \in V(H) \setminus \{x\}:\codeg(x,y) \geq \l\}$. When $\l=1$ we omit the superscript and recover the usual notation for a neighbourhood. Let $H_1$ and $H_2$ be hypergraphs. The union $H_1 \uplus H_2$ of $H_1$ and $H_2$ is the hypergraph with vertex set $V(H_1) \cup V(H_2)$ and edge multiset $E(H_1 \uplus H_2)$ defined by $\mu_{E(H_1 \uplus H_2)}(S)=\mu_{E(H_1)}(S)+\mu_{E(H_2)}(S)$ for each subset $S$ of $V(H_1) \cup V(H_2)$, where $\mu_{E(H_1)}(S)$ denotes the number of times that $S$ occurs in the multiset $E(H_1)$ and $\mu_{E(H_2)}(S)$ denotes the number of times that $S$ occurs in the multiset $E(H_2)$. Denote by $H_1 - H_2$ the hypergraph with vertex set $V(H_1)$ and edge multiset $E(H_1 - H_2)$ defined by $\mu_{E(H_1 - H_2)}(S)=\max\{\mu_{E(H_1)}(S)-\mu_{E(H_2)}(S),0\}$ for each subset $S$ of $V(H_1)$. If $H$ is a hypergraph, the \emph{underlying multigraph} $G$ of $H$ is the multigraph $G$ with vertex set $V(H)$ such that $\codeg_G(x,y)=\codeg_H(x,y)$ for all distinct $x,y \in V(H)$. On the other hand, if $H$ is $\l$-linear, we define the \emph{$\l$-defect} of $H$ to be the multigraph $G$ with vertex set $V(H)$ such that $\codeg_G(x,y)=\l-\codeg_H(x,y)$ for all distinct $x,y \in V(H)$.
We define a \emph{matching} as a graph in which each vertex has degree 1. We allow the possibility of a trivial matching with empty vertex and edge sets.

An edge decomposition, hereafter simply a \emph{decomposition}, of a multigraph $G$ is a multiset $\mathcal{F}$ of simple subgraphs of $G$ such that each edge of $G$ is in exactly one of the graphs in $\mathcal{F}$ (so that, for all distinct $x,y \in V(G)$, there are exactly $\codeg_G(x,y)$ graphs in $\mathcal{F}$ in which $x$ and $y$ are adjacent). We say the decomposition is an \emph{$F$-decomposition} if $F$ is a graph such that each graph in $\mathcal{F}$ is isomorphic to $F$. For a nonempty graph $F$, we define $\gcd(F)=\gcd\{\deg_F(x):x \in V(F)\}$  and we say a multigraph $G$ is \emph{$F$-divisble} if $|E(F)|$ divides $|E(G)|$ and $\gcd(F)$ divides $\deg_G(x)$ for each $x \in V(G)$. It is easy to see that if there is an $F$-decomposition of a multigraph $G$, then $G$ is necessarily $F$-divisible. Note that an $\ell$-uniform hypergraph with underlying multigraph $G$ can also be viewed as a $K_\ell$-decomposition of $G$. Our approach to proving Theorem~\ref{T:boundAchieved} and the latter part of Theorem~\ref{T:boundNotAchieved}(iii) is based on the fact that large $F$-divisible graphs of high minimum degree have $F$-decompositions. We direct the reader to \cite{GloKuhLoMonOst} or \cite{Kee} for a discussion of this problem and stronger results than we require here. The following result (see \cite[Corollary~1.6(i)]{GloKuhLoMonOst}, for example) will be sufficient for our purposes.\pagebreak

\begin{theorem}\label{T:denseSimpleFDecomps}
Let $F$ be a graph. There is a positive constant $\gamma=\gamma(F)$ and an integer $M=M(F)$ such that, for all integers $m > M$, each $F$-divisible graph of order $m$ with minimum degree at least $(1-\gamma)m$ has an $F$-decomposition.
\end{theorem}

We require an analogue of Theorem~\ref{T:denseSimpleFDecomps} that gives decompositions of $\l$-linear multigraphs. Furthermore, we are interested in cases where $\l$ is arbitrary, rather than being fixed as the order of the host multigraph grows. We first use Theorem~\ref{T:denseSimpleFDecomps} to prove an analogous result for $\l$-linear multigraphs where $\l$ is fixed, and then in turn use this result to obtain the tool we require. In proving the intermediate result we make use of the following special case of \cite[Theorem 1(ii)]{KanTok}.

\begin{lemma}[\textup{\cite{KanTok}}]\label{L:KanTok}
Let $C$ be a graph with vertex set $V$, let $s \geq 2$ be an integer and, for each $x \in V$, let $a_x$ be an integer such that $1 \leq a_x \leq s$. If $\sum_{x \in V}a_x \equiv 0 \mod{2}$, $|V(C)| \geq (s+1)^2$, and $\delta(C) \geq \frac{s}{s+1}|V|$, then there is a subgraph $C'$ of $C$ such that $\deg_{C'}(x)=a_x$ for each $x \in V$.
\end{lemma}

\begin{lemma}\label{T:denseFDecompsConstLambda}
Let $\lambda$ be a positive integer and let $F$ be a graph. There is a positive constant $\gamma=\gamma(\lambda,F)$ and integer $M=M(\lambda,F)$ such that, for all integers $m>M$, each $F$-divisible $\lambda$-linear multigraph of order $m$ with minimum degree at least $(1-\gamma)\l m$ has an $F$-decomposition.
\end{lemma}

\begin{proof}
Fix a graph $F$ and let $s=2|E(F)|\gcd(F)$. We may assume $F$ has no isolated vertices. We prove the result by induction on $\lambda$. By Theorem~\ref{T:denseSimpleFDecomps} there is a positive constant $\gamma_1 < \frac{1}{s}$ and integer $M_1$ such that, for all integers $m>M_1$, each $F$-divisible linear multigraph of order $m$ with minimum degree at least $(1-\gamma_1) m$ has an $F$-decomposition. So the result holds for $\lambda=1$. Now fix a $\lambda \geq 2$ and suppose that there is a positive constant $\gamma_{\l-1}$ and integer $M_{\l-1}$ such that, for all integers $m>M_{\l-1}$, each $F$-divisible $(\l-1)$-linear multigraph of order $m$ with minimum degree at least $(1-\gamma_{\l-1})(\l-1) m$ has an $F$-decomposition. Let $\gamma_\l=\min\{\frac{1}{2\l}\gamma_1,\frac{1}{2}\gamma_{\l-1}\}$. We will show that for sufficiently large $m$ each $F$-divisible $\lambda$-linear multigraph of order $m$ with minimum degree at least $(1-\gamma_\l)\l m$ has an $F$-decomposition. For the rest of the proof we tacitly assume that $m$ is large whenever required.

Let $G$ be an $F$-divisible $\l$-linear multigraph of order $m$ with $\delta(G) \geq (1-\gamma_{\l})\l m$ and let $V=V(G)$. Let $C$ be the spanning subgraph of $G$ such that, for all distinct $x,y \in V$, $\codeg_{C}(x,y)=1$ if $\codeg_{G}(x,y) = \l$ and $\codeg_{C}(x,y)=0$ otherwise. Note that the $\l$-defect of $G$ has maximum degree at most $\gamma_{\l}\l m-\l$ and hence
\begin{equation}\label{E:minDegC}
\delta(C) \geq (m-1)-(\gamma_{\l}\l m-\l) > \left(1-\gamma_{\l}\l\right) m \geq \left(1-\tfrac{1}{2}\gamma_1\right) m
\end{equation}
where the last inequality follows because $\gamma_\l \leq \frac{1}{2\l}\gamma_1$.

For each $x \in V$, let $a_x$ be the least positive integer such that $\deg_{C}(x)+a_x \equiv 0 \mod{s}$ and note that $\sum_{x \in V}a_x$ is even because $s$ is even. So there is a subgraph $C'$ of $C$ such that $\deg_{C'}(x) =a_x$ for each $x \in V$ by Lemma~\ref{L:KanTok} because $(s+1)^2=O(1)$ and $\delta(C) \geq (1-\frac{1}{2}\gamma_1) m > \frac{s}{s+1}m$. Let $G'$ be the $2$-linear multigraph $C \uplus C'$ and note that $G'$ is a submultigraph of $G$ by the definition of $C$. We will complete the proof by showing that both $G'$ and $G-G'$ have $F$-decompositions. Observe that $G'$ and $G-G'$ are $F$-divisible because $G$ is $F$-divisible and $\deg_{G'}(x) \equiv 0 \mod{s}$ for each $x \in V$.

We first show that $G-G'$ has an $F$-decomposition. By the definitions of $G'$ and $C$, we have that $G-G'$ is $(\l-1)$-linear. Also, $\delta(G-C) \geq \frac{\l-1}{\l}\delta(G) \geq (1-\gamma_\l)(\l-1) m$ and thus, because $\Delta(C') \leq s = O(1)$ and $\gamma_\l \leq \frac{1}{2}\gamma_{\l-1}$, we have
\[\delta(G-G') \geq (1-\gamma_\l)(\l-1) m-O(1) >  (1-\gamma_{\l-1})(\l-1) m.\]
So $G-G'$ has an $F$-decomposition by the definition of $\gamma_{\l-1}$.

It remains to show that $G'$ has an $F$-decomposition. We begin by deleting copies of $F$ from $G'$ until we obtain a (simple) graph. Let $G'_0=G'$ and $t=|E(C')|$ and note that $t \leq \frac{1}{2}ms$ since $a_x \leq s$ for each $x \in V$. We define a sequence $G'_0,\ldots,G'_t$ of multigraphs by iterating the following procedure for $i=0,1,\ldots,t-1$.
\begin{itemize}
    \item[(1)]
If $G'_{i}$ is a (simple) graph then let $G'_{i+1}=G'_i$.
    \item[(2)]
Otherwise, choose any pair of distinct vertices $y,z \in V$ such that $\codeg_{G'_i}(y,z)=2$. Let $G'_{i+1}=G'_i-F_i$ for some copy $F_i$ of $F$ in $G'_i$ such that $yz \in E(F_i)$, and $V(F_i) \setminus \{y,z\} \subseteq S_{i}$ where $S_i=\{x \in V \setminus \{y,z\}:|N_{G'_i}(x)| \geq (1-\frac{3}{4}\gamma_1)m\}$.
\end{itemize}
To see that we can indeed define such a sequence, suppose we have successfully defined $G'_i$ for some $i \in \{0,\ldots,t-1\}$. Because $|N_{G'}(x)| = \deg_C(x) > (1-\frac{1}{2}\gamma_1) m$ for each $x \in V$ and $G'_i$ was obtained from $G'$ by deleting at most $t=O(m)$ copies of $F$, it can be seen that the size of the set $S_i$ defined in (2) is $m-O(1)$. Thus, by the definition of $S_i$ and because $\gamma_1<\frac{1}{s} \leq \frac{1}{|V(F)|}$, there is a set of $|V(F)|-2$ vertices in $S_i$ that are pairwise adjacent in $G'_i$ and hence a suitable choice of graph $F_i$ exists. So we can define the sequence $G'_0,\ldots,G'_t$ and it suffices to show that $G'_t$ has an $F$-decomposition.

Note that $G'_t$ is $F$-divisible since $G'$ was. Also, $G'_t$ is a (simple) graph by the definition of $G'_0,\ldots,G'_t$ and the fact that $t = |E(C')|$, Finally, by the definitions of $G'_0,\ldots,G'_t$ and $S_0,\ldots,S_{t-1}$, we have that
\[\delta(G'_t) \geq (1-\tfrac{3}{4}\gamma_1)m - (\Delta(C')+1)\Delta(F) = (1-\tfrac{3}{4}\gamma_1)m - O(1) > (1-\gamma_1) m.\]
So $G'_t$ has an $F$-decomposition by the definition of $\gamma_1$ and the proof is complete.
\end{proof}

We can easily extend this result to cover cases where $\l$ is not fixed by strengthening the condition on minimum degree to a condition on minimum $\l$-neighbourhood size. This will be sufficient for our purposes.

\begin{lemma}\label{T:denseFDecompsArbLambda}
Let $F$ be a graph. There is a positive constant $\gamma=\gamma(F)$ and an integer $M=M(F)$ such that, for all integers $m>M$ and $\lambda \geq 1$, each $F$-divisible $\l$-linear multigraph $G$ of order $m$ such that $|N^\l_G(x)| \geq (1-\gamma) m$ for each $x \in V(G)$ has an $F$-decomposition.
\end{lemma}

\begin{proof}
Fix a graph $F$ and let $s=2|E(F)|\gcd(F)$. For each $i \in \{1,\ldots,2s\}$, by Lemma~\ref{T:denseFDecompsConstLambda} there is a positive constant $\gamma_i$ and integer $M_i$ such that, for all integers $m>M_i$, each $F$-divisible $i$-linear multigraph of order $m$ with minimum degree at least $(1-\gamma_i)i m$ has an $F$-decomposition. Let $\gamma=\gamma(F)=\min\{\gamma_1,\ldots,\gamma_{2s}\}$ and $M=M(F)=\max\{M_1,\ldots,M_{2s}\}$. We will show by induction on $\l$ that the result holds with this choice of $\gamma$ and $M$ for all integers $\lambda \geq 1$. Because $\deg_G(x) \geq \l|N^\l_G(x)|$ for any graph $G$ and each $x \in V(G)$, the result holds when $\l \in \{1,\ldots,2s\}$ by our choice of $\gamma$ and $M$.

Fix a positive integer $\l \geq 2s+1$ and suppose that, for each $i \in \{1,\ldots,\l-1\}$, each $F$-divisible $i$-linear multigraph of order $m>M$ with minimum $i$-neighbourhood size at least $(1-\gamma) m$ has an $F$-decomposition. Let $G$ be an $F$-divisible $\l$-linear multigraph of order $m>M$ such that $|N^{\l}_G(x)| \geq (1-\gamma) m$ for each $x \in V(G)$. Let $G''$ be the $s$-linear submultigraph of $G$ with vertex set $V(G)$ and, for all distinct $x,y \in V(G)$, $\codeg_{G''}(x,y)=s$ if $\codeg_{G}(x,y)>s$ and $\codeg_{G}(x,y)=0$ otherwise. Now $G''$ is $F$-divisible by our choice of $s$ and has $\delta(G'') \geq (1-\gamma)s m$ because $\deg_{G''}(x) \geq s |N^{\l}_G(x)|$ for each $x \in V(G)$. So $G''$ has an $F$-decomposition by our definitions of $\gamma$ and $M$.  Let $G'=G-G''$, let $\l'=\l-s$ and note that $\l'>s$. We have that $G'$ is $F$-divisible because both $G$ and $G''$ are. Also, $G'$ is $\l'$-linear by the definition of $G''$ and because $\l'>s$. Further, for each $x \in V(G)$, $|N^{\l'}_{G'}(x)| = |N^{\l}_{G}(x)| \geq (1-\gamma) m$. So $G'$ has an $F$-decomposition by our inductive hypothesis. Thus $G$ has an $F$-decomposition because both $G'$ and $G''$ do.
\end{proof}

Most frequently, we will apply Lemma~\ref{T:denseFDecompsArbLambda} with $F=K_\ell$ to obtain an $\ell$-uniform hypergraph with a specified underlying multigraph $G$. Note that the condition that $G$ is $K_\ell$-divisible is equivalent to $|E(G)| \equiv 0 \mod{\binom{\ell}{2}}$ and $\deg_G(x) \equiv 0 \mod{\ell-1}$ for each $x \in V(G)$.

The other key ingredient in our proof of Theorem~\ref{T:boundAchieved} is a result which allows us to construct an $\ell$-uniform hypergraph $H$ with a specified degree sequence such that the union of $H$ with a given multigraph is $\l$-linear. The problem of whether there exists an $\ell$-uniform $\l$-linear hypergraph with a specified degree sequence is hard in general, but it becomes easy when the total degree is large compared to the maximum degree.

\begin{lemma}\label{L:graphRealisation}
Let $\l \geq 1$ and $\ell \geq 2$ be integers, let $V$ be a set of vertices, let $A$ be a $\l$-linear multigraph such that $V(A) \subseteq V$ and, for each $x \in V$, let $r_x$ be a nonnegative integer. There is an $\ell$-uniform hypergraph $H$ on vertex set $V$ such that $H \uplus A$ is $\l$-linear and $\deg_H(x)=r_x$ for each $x \in V$, if
\begin{itemize}
    \item[\textup{(i)}]
$\sum_{x \in V}r_x = \ell n^\dag$ for some positive integer $n^\dag$; and
    \item[\textup{(ii)}]
$n^\dag \geq 2qr$ where $q=1+\lfloor\frac{1}{\l}\max\{r_x(\ell-1)+\deg_A(x):x\in V^+\}\rfloor$, $V^+=\{x \in V:r_x >0\}$, and $r=\max\{r_x:x\in V\}$.
\end{itemize}
\end{lemma}

\begin{proof}
Because we can add isolated vertices as required to the final hypergraph produced, we may assume that $V=V^+$. We first show that there exists a (not necessarily $\l$-linear) $\ell$-uniform hypergraph $M$ on vertex set $V$ such that $\deg_M(x)=r_x$ for each $x \in V$. Note that such a hypergraph will have $n^\dag$ edges and that $n^\dag > r$ by (ii). So, using the definition of $r$, we have that $\sum_{x \in V'}r_x \leq r|V'| < n^\dag|V'|$ for any subset $V'$ of $V$. Thus, by the well-known Gale-Ryser theorem (see \cite[Theorem 4.3.18]{Wes} for example), there is a bipartite graph $B$ with parts $U$ and $V$ where $|U|=n^\dag$, $\deg_B(x)=\ell$ for each $x \in U$ and $\deg_B(x)=r_x$ for each $x \in V$. Then the hypergraph $M$ on vertex set $V$ with edge multiset $\{N_B(x):x \in U\}$ is $\ell$-uniform and has $\deg_M(x)=r_x$ for each $x \in V$.

Now consider the hypergraph $M \uplus A$. For each $x \in V$, we use $N^\l_{M \uplus A}[x]$ to denote $\{x\} \cup N^\l_{M \uplus A}(x)$ and note that $|N^\l_{M \uplus A}[x]| \leq q$ for each $x \in V$. If $M \uplus A$ is $\l$-linear, then we are done. Otherwise, let $u$ and $v$ be vertices in $V$ such that $\codeg_{M \uplus A}(u,v) \geq \l+1$. We will show we can modify $M$ to obtain an $\ell$-uniform hypergraph $M'$ such that $\deg_{M'}(x)=\deg_{M}(x)$ for each $x \in V$,  $\codeg_{M' \uplus A}(u,v)=\codeg_{M \uplus A}(u,v)-1$, and $\codeg_{M' \uplus A}(x,y) \leq \max\{\codeg_{M \uplus A}(x,y),\l\}$ for all distinct $x,y \in V$. This will suffice to complete the proof since by repeating this procedure we will eventually obtain a $\l$-linear hypergraph $H$ with the required properties.

Let $E_{uv}$ be an edge of $M$ such that $\{u,v\} \subseteq E_{uv}$. Such an $E_{uv}$ exists because $\codeg_{M \uplus A}(u,v) \geq \l+1$ yet $A$ is $\l$-linear. Let
\begin{align*}
  S_1 &= \{E \in E(M): E \cap N^\l_{M \uplus A}[u] \neq \emptyset\} \\
  S_2 &= \{E \in E(M): E \subseteq \textstyle{\bigcup_{y \in E_{uv} \setminus \{u\}}N^\l_{M \uplus A}[y]}\}.
\end{align*}
Observe that $|S_1| \leq qr$ because $|N^\l_{M \uplus A}[u]| \leq q$. Also $\bigcup_{y \in E_{uv} \setminus \{u\}}N^\l_{M \uplus A}[y] \leq (\ell-1)q$ and hence $|S_2| \leq \frac{\ell-1}{\ell}qr < qr$. Thus, by (ii), there is an edge of $M$ that is in neither $S_1$ nor $S_2$. Since this edge is not in $S_2$, it contains a vertex not in $\bigcup_{y \in E_{uv} \setminus \{u\}}N^\l_{M \uplus A}[y]$. Call the vertex $w$ and the edge $E_w$. Note that $u,v \notin E_w$ since $E_w \notin S_1$. Also, $w \notin E_{uv}$ by our choice of $w$.

Now, let $M'$ be the hypergraph on vertex set $V$ obtained from $M$ by replacing (one copy of) the edges $E_{uv}$ and $E_w$ with
(one copy of) the edges $(E_{uv} \setminus \{u\}) \cup \{w\}$ and $(E_w \setminus \{w\}) \cup \{u\}$. Recalling that $w \notin E_{uv}$ and $u \notin E_w$, we have that $M'$ is $\ell$-uniform and that $\deg_{M'}(x)=\deg_{M}(x)$ for each $x \in V$. Furthermore, for all distinct $x,y \in V$,
\[\codeg_{M'}(x,y)=
\left\{
  \begin{array}{ll}
    \codeg_{M}(x,y)-1\  & \hbox{if $\{x,y\} \in \{\{u,z\}: z\in E^*_{uv}\} \cup \{\{w,z\}: z\in E^*_w\}$} \\
    \codeg_{M}(x,y)+1 & \hbox{if $\{x,y\} \in \{\{w,z\}: z\in E^*_{uv}\} \cup \{\{u,z\}: z\in E^*_w\}$} \\
    \codeg_{M}(x,y) & \hbox{otherwise}
  \end{array}
\right.
\]
where $E^*_{uv}=E_{uv} \setminus (E_w \cup \{u\})$ and $E^*_{w}=E_{w} \setminus (E_{uv} \cup \{w\})$. Thus, because $v \in E^*_{uv}$, we have $\codeg_{M' \uplus A}(u,v)=\codeg_{M \uplus A}(u,v)-1$. For each $z \in E^*_{uv}$, we have $\codeg_{M' \uplus A}(w,z) \leq \l$ because $w \notin N^\l_{M \uplus A}[z]$ implies $\codeg_{M \uplus A}(w,z)< \l$. For each $z\in E^*_w$, we have $\codeg_{M' \uplus A}(u,z) \leq \l$ because $E_w \notin S_1$ implies $\codeg_{M \uplus A}(u,z)< \l$. Thus $\codeg_{M' \uplus A}(x,y) \leq \max\{\codeg_{M \uplus A}(x,y),\l\}$ for all distinct $x,y \in V$. So $M'$ has the properties we require and the proof is complete.
\end{proof}

\section{Achieving \texorpdfstring{$\boldsymbol{\lfloor A^k_\l(m,n) \rfloor}$}{A(m,n)}} \label{S:achievingA}

Over the course of this section and the next one, we will prove Theorem~\ref{T:boundAchieved}. For many of our results, we will be operating in a regime in which $k$ is fixed, $m$ is large (as a function of $k$) and $\l$ is arbitrary. We will employ asymptotic notation with respect to this regime.

In this section we show that $Z_{2,\l+1}(m,n)=\lfloor A^k_\l(m,n) \rfloor$ in the cases that Theorem~\ref{T:boundAchieved} specifies. Our first step in this direction is to translate an optimum for $\A_\l(m,n)$ with two nonzero real values given by Lemma~\ref{L:aBoundWithEq} into integer numbers of edges of sizes $k$ and $k+1$ for which we can realise a $\l$-linear hypergraph meeting the total degree bound $\lfloor A^k_\l(m,n) \rfloor$. It turns out that we only need take a ceiling and a floor to accomplish this, as we formalise with the following definitions and Lemma~\ref{L:zBoundFloor}.

Throughout this section we define
\begin{equation}\label{E:nDefn}
n_k=\left\lceil \tfrac{k+1}{2}n-\tfrac{\l}{k}\tbinom{m}{2} \right\rceil \quad\text{and}\quad n_{k+1} = \left\lfloor \tfrac{\l}{k}\tbinom{m}{2}-\tfrac{k-1}{2}n \right\rfloor
\end{equation}
where the values of $\l$, $m$ and $n$ will always be clear from context.  Further, we let $d$ be the element of $\{0,\ldots,k-1\}$ such that $\l\binom{m}{2}-\binom{k}{2}n \equiv d \mod{k}$. This implies that
\begin{equation}\label{E:nkExact}
n_k=\tfrac{k+1}{2}n-\tfrac{\l}{k}\tbinom{m}{2}+\tfrac{d}{k} \quad\text{ and }\quad n_{k+1} = \tfrac{\l}{k}\tbinom{m}{2}-\tfrac{k-1}{2}n -\tfrac{d}{k}.
\end{equation}
Thus, $n_k+n_{k+1}=n$ and, using both equalities in \eqref{E:nkExact},
\begin{equation}\label{E:edgeSumEasy}
\tbinom{k}{2}n_k+\tbinom{k+1}{2}n_{k+1}=\l\tbinom{m}{2}-d.
\end{equation}

The following simple lemma explains the reason for these definitions.

\begin{lemma}\label{L:zBoundFloor}
Let $k$, $\l$, $m$ and $n$ be positive integers. If there is a $\l$-linear hypergraph $H$ of order $m$ with $n_k$ edges of size $k$ and $n_{k+1}$ edges of size $k+1$, then $Z_{2,\l+1}(m,n) = \lfloor A_\l^k(m,n) \rfloor$ and the $\l$-defect of $H$ has $d$ edges.
\end{lemma}

\begin{proof}
By Lemmas~\ref{L:obeyLP} and \ref{L:aBoundWithEq} we have $Z_{2,\l+1}(m,n) \leq \lfloor A_\l^k(m,n) \rfloor$. Let $H$ be a $\l$-linear hypergraph of order $m$ with $n_k$ edges of size $k$ and $n_{k+1}$ edges of size $k+1$. Then, using \eqref{E:nkExact}, we can calculate that the total degree of $H$ is $A_\l^k(m,n)-\tfrac{d}{k}$. Thus, because $d \in \{0,\ldots,k-1\}$, the total degree of $H$ is $\lfloor A_\l^k(m,n) \rfloor$. It is clear from \eqref{E:edgeSumEasy} that the $\l$-defect of $H$ has $d$ edges.
\end{proof}

In view of Lemma~\ref{L:zBoundFloor}, in the remainder of this section we only need to show the existence of appropriate $\l$-linear hypergraphs in the situations specified by Theorem~\ref{T:boundAchieved}. We first show that we can obtain these hypergraphs for most of the interval prescribed by Theorem~\ref{T:boundAchieved}. The approach we take to accomplish this has similarities to one used in \cite{CarYus}.

\begin{lemma}\label{L:middleRange}
Let $k \geq 2$ be a fixed integer and $\e>0$ be a fixed constant. For each sufficiently large integer $m$, we have $Z_{2,\l+1}(m,n) = \lfloor A_\l^k(m,n) \rfloor$ for all positive integers $\l$ and $n$ such that \[(1+\e)\l\tbinom{m}{2}/\tbinom{k+1}{2} \leq n \leq (1-\e)\l\tbinom{m}{2}/\tbinom{k}{2}.\]
\end{lemma}

\begin{proof}
Assume that $m$ is a large integer and that $\l$ and $n$ are positive integers satisfying the hypotheses of the lemma.
Observe using \eqref{E:nkExact} that $n_k \geq \e\frac{\l}{k}\binom{m}{2}$ because $n \geq (1+\e)\l\binom{m}{2}/\binom{k+1}{2}$ and that $n_{k+1} \geq \e\frac{\l}{k}\binom{m}{2}-\frac{d}{k}$ because $n \leq (1-\e)\l\binom{m}{2}/\binom{k}{2}$. By Theorem~\ref{T:denseSimpleFDecomps} we can choose a constant $h=h(k)$ such that $h \equiv 1 \mod{k(k^2-1)}$ and $h$ is large enough that $\e\binom{h}{2} \geq 2k$ and $K_h$ has a $K_\ell$-decomposition for each $\ell \in \{k,k+1\}$. Let $F$ be the vertex-disjoint union of one copy of $K_h$, one copy of $K_k$ and one copy of $K_{k+1}$. Our overall strategy is to first decompose $\l K_m$ into a large number, $b$, of copies of $F$ together with a small matching and a small number $n'_\ell$ of copies of $K_\ell$ for each $\ell \in \{k,k+1\}$. We will then further decompose the copies of $F$ in order to produce the required $\l$-linear hypergraph of order $m$ with $n_\ell$ of edges of size $\ell$ for each $\ell \in \{k,k+1\}$. To show this can be done we first define $b$, $n'_k$ and $n'_{k+1}$ along with some other important quantities.
\begin{itemize}
    \item
Let $q_\ell$ be the integer $\binom{h}{2}/\binom{\ell}{2}$ for each $\ell \in \{k,k+1\}$.
    \item
Let $b$ and $e$ be the integers such that
\begin{equation}\label{E:ePartition}
\l\tbinom{m}{2}-d=b\left(\tbinom{h}{2}+\tbinom{k+1}{2}+\tbinom{k}{2}\right)+e
\end{equation}
and $\tbinom{h}{2} \leq e  < 2\tbinom{h}{2}+\tbinom{k+1}{2}+\tbinom{k}{2}$. Observe that $b < \l\binom{m}{2}/\binom{h}{2}-1$ and hence $n_k,n_{k+1} \geq 2b$ by the lower bounds on $n_k$ and $n_{k+1}$ established above and the fact that $\e\binom{h}{2} \geq 2k$.
    \item
Let $b_k$ and $n'_k$ be the integers such that $n_k=b+b_kq_k+n'_k$ and $0 \leq n'_k < q_k$. Note that $b_k \geq 0$ since $n_k \geq 2b$.
    \item
Let $b_{k+1}=b-b_k$ and $n'_{k+1}=(e-n'_k\tbinom{k}{2})/\binom{k+1}{2}$. Observe that $n'_{k+1} \geq 0$ since $e \geq \binom{h}{2}$ and $n'_k\tbinom{k}{2} < q_k\tbinom{k}{2} \leq \binom{h}{2}$.
\end{itemize}
Note that $q_k$, $q_{k+1}$, $e$, $n'_k$ and $n'_{k+1}$ are all $O(1)$. Substituting $b\binom{h}{2}=b_kq_k\binom{k}{2}+b_{k+1}q_{k+1}\binom{k+1}{2}$ and $e=n'_k\binom{k}{2}+n'_{k+1}\binom{k+1}{2}$ in \eqref{E:ePartition} we have
\begin{equation}\label{E:fancyPartition}
\l\tbinom{m}{2}-d=(b+b_kq_k+n'_k)\tbinom{k}{2}+(b+b_{k+1}q_{k+1}+n'_{k+1})\tbinom{k+1}{2}.
\end{equation}
Now, since $n_k=b+b_kq_k+n'_k$, by equating the left side of \eqref{E:edgeSumEasy} with the right side of \eqref{E:fancyPartition} we see that $n_{k+1}=b+b_{k+1}q_{k+1}+n'_{k+1}$. It follows that $n'_{k+1}$ is an integer. Furthermore, $b_{k+1} \geq 0$ because $n_{k+1} \geq 2b$, $b=\Theta(m^2)$ and $n'_{k+1}<e=O(1)$.

Let $V$ be a set of $m$ vertices and let $D$ be a matching with $d$ edges such that $V(D) \subseteq V$. By Lemma~\ref{L:zBoundFloor} it suffices to find a hypergraph with $n_k$ edges of size $k$, $n_{k+1}$ edges of size $k+1$ and underlying multigraph $\l K_V-D$. We obtain this as a union of two hypergraphs $H'$ and $H''$. Let $H'$ be a hypergraph with vertex set $V$ such that $H'$ has $n'_\ell$ edges of size $\ell$ for each $\ell \in \{k,k+1\}$ and the edges of $H' \uplus D$ are pairwise vertex-disjoint (note that $2d+kn'_k+(k+1)n'_{k+1}=O(1)$).

Let $G''$ be the $\l$-defect of $H' \uplus D$ and observe that $|E(G'')|=\l\binom{m}{2}-d-e$ using the definition of $n'_{k+1}$.  Now $\gcd(F)=1$, $|E(G'')|=b|E(F)|$ by \eqref{E:ePartition}, and $|N^\l_{G''}(x)| \geq m-1-k = m-O(1)$ for each $x \in V$. Thus by Lemma~\ref{T:denseFDecompsArbLambda} there is a decomposition of $G''$ into $b$ copies of $F$ and hence a decomposition of $G''$ into $b$ copies of $K_h$, $b$ copies of $K_k$ and $b$ copies of $K_{k+1}$. Let $H''$ be the hypergraph formed by decomposing each of $b_k$ of these copies of $K_h$ into $q_k$ edges of size $k$ and each of the remaining $b_{k+1}$ copies of $K_h$ into $q_{k+1}$ edges of size $k+1$. Then $H''$ has underlying multigraph $G''$, exactly $n_k-n'_k$ edges of size $k$ and exactly $n_{k+1}-n'_{k+1}$ edges of size $k+1$ (recall that $n_\ell=b+b_\ell q_\ell+n'_\ell$ for each $\ell \in \{k,k+1\}$). So $H' \uplus H''$ is a hypergraph with $n_k$ edges of size $k$, $n_{k+1}$ edges of size $k+1$ and underlying multigraph $\l K_V-D$.
\end{proof}

We must work a little harder to obtain the hypergraphs required when $n$ is near the extreme ends of the interval prescribed by Theorem~\ref{T:boundAchieved}. At the lower end of this interval $n_k$ is small while $n_{k+1}$ is large and so we first construct a $k$-uniform hypergraph with $n_{k}$ edges and small maximum neighbourhood size, and then use Lemma~\ref{T:denseFDecompsArbLambda} to obtain a $(k+1)$-uniform hypergraph with $n_{k+1}$ edges whose union with the first hypergraph is $\l$-linear. Constructing the first of these hypergraphs is a two step process involving first applying Lemma~\ref{T:denseFDecompsArbLambda} to place all but approximately $\frac{\beta m}{k}$ of the edges on a small number of vertices and then using Lemma~\ref{L:graphRealisation} to carefully distribute the rest of the edges across the vertices. At the upper end of the interval the situation is reversed: $n_{k+1}$ is small while $n_k$ is large and we employ a similar strategy with the roles of edges of size $k$ and edges of size $k+1$ exchanged.

\begin{lemma}\label{L:aBoundAchieved}
Let $k \geq 2$ be a fixed integer. For each sufficiently large integer $m$, we have $Z_{2,\l+1}(m,n) = \lfloor A_\l^k(m,n) \rfloor$ for all positive integers $\l$ and $n$ such that
\[\left(\l\tbinom{m}{2}+\beta m\right)/\tbinom{k+1}{2}+4\bigl(\tfrac{k^2}{\l}+1\bigr) \leq n \leq \left(\l\tbinom{m}{2}-\tfrac{\alpha k}{k+1}m\right)/\tbinom{k}{2}-4\bigl(\tfrac{k(k-1)}{\l}+1\bigr).\]
\end{lemma}

\begin{proof}
Assume that $m$ is a large integer and that $\l$ and $n$ are positive integers satisfying the condition in the lemma statement.
Let $V$ be a set of $m$ vertices. By Lemma~\ref{T:denseFDecompsArbLambda} we can fix a constant $\gamma=\gamma(k)<\frac{1}{3k^2}$ such that, for each $\ell \in \{k,k+1\}$, there is a $K_\ell$-decomposition of any $K_\ell$-divisible $\l$-linear multigraph of order $m$ with minimum $\l$-neighbourhood size at least $(1-\gamma)m$. Also by Lemma~\ref{T:denseFDecompsArbLambda} we can choose a sufficiently large constant $h=h(k)$ such that $h \equiv 1 \mod{k(k^2-1)}$ and $K_h$ has a $K_\ell$-decomposition for each $\ell \in \{k,k+1\}$. Choose a subset $V'$ of $V$ such that $|V'|$ is the least integer for which $|V'| \geq \frac{1}{2} \gamma m$ and $\l(|V'|-1) \equiv 0 \mod{h(h-1)}$. Then, since $m$ is large, there is an $h$-uniform hypergraph $H_h$ whose underlying multigraph is $\l K_{V'}$ by Lemma~\ref{T:denseFDecompsArbLambda}. Note that $|V'|=\frac{1}{2} \gamma m+O(1)$. Let $D$ be a matching with $d$ edges whose vertex set is a subset of $V \setminus V'$.

Let $\epsilon= \epsilon(k)= \frac{1}{3(k+1)}\gamma^2$. By Lemma~\ref{L:middleRange}, we may assume that either
\[\frac{\l\binom{m}{2}+\beta m}{\binom{k+1}{2}}+4\bigl(\tfrac{k^2}{\l}+1\bigr) \leq n \leq \frac{(1+\epsilon)\l\binom{m}{2}}{\binom{k+1}{2}} \text{ or } \frac{(1-\epsilon)\l\tbinom{m}{2}}{\binom{k}{2}} \leq n \leq \frac{\l\binom{m}{2}-\frac{\alpha k}{k+1}m}{\binom{k}{2}}-4\bigl(\tfrac{k(k-1)}{\l}+1\big).\]
We now divide the proof into Cases 1 and 2 according to whether the former or the latter of these holds, respectively.

\textbf{Case 1.} By the bounds on $n$ in this case and \eqref{E:nkExact} we have
\begin{equation} \label{E:smallEdgesLB}
\tfrac{\beta m+d}{k}+2(k+1)\bigl(\tfrac{k^2}{\l}+1\bigr) \leq n_k \leq \tfrac{\epsilon \l}{k}\tbinom{m}{2}+\tfrac{d}{k}.
\end{equation}
We will form the desired hypergraph $H$ as the union of $k$-uniform hypergraphs $H'_{k}$ and $H''_k$ and a $(k+1)$-uniform hypergraph $H_{k+1}$. All but approximately $\frac{\beta m}{k}$ of the required $n_k$ edges of size $k$ will be used in $H'_k$, and then $H''_k$ will contain the remaining edges of size $k$, carefully distributed so as to make the $\l$-defect of $H'_k \uplus H''_k \uplus D$ be $K_{k+1}$-divisible. We will then apply Lemma~\ref{T:denseFDecompsArbLambda} to this defect to obtain $H_{k+1}$. We proceed to describe each of these hypergraphs and justify that they do indeed exist. \smallskip

\noindent \textbf{$\boldsymbol{H'_k}$:} This will be a (possibly empty) $k$-uniform $\l$-linear hypergraph on vertex set $V'$ such that $\deg_{H'_k}(x) \equiv 0 \mod{k}$ for each $x \in V'$ and $H'_k$ has $n'_k$ edges, where $n'_k$ is the greatest integer such that $n'_k \equiv 0 \mod{\binom{h}{2}/\binom{k}{2}}$ and
\[n_k-n'_k \geq \tfrac{\beta m+d}{k}+2(k+1)\bigl(\tfrac{k^2}{\l}+1\bigr).\]
Note that $n'_k \geq 0$ by \eqref{E:smallEdgesLB}. We form $H'_k$ by decomposing each of $n'_k\binom{k}{2}/\binom{h}{2}$ edges of $H_h$ into $\binom{h}{2}/\binom{k}{2}$ edges of size $k$ and deleting the remaining edges of $H_h$. There are sufficiently many edges in $H_h$ to accomplish this because it can be established that $n_k\binom{k}{2} \leq \l\binom{|V'|}{2}$ using \eqref{E:smallEdgesLB}, $|V'| \geq \frac{1}{2} \gamma m$ and $\epsilon = \frac{1}{3(k+1)}\gamma^2$. For each $x \in V$, we
have $\deg_{H'_k}(x) \equiv 0 \mod{k}$ because $\deg_{H'_k}(x)$ is a multiple of $\frac{h-1}{k-1}$ and $h \equiv 1 \mod{k(k-1)}$. \smallskip

\noindent \textbf{$\boldsymbol{H''_k}$:} This will be a $k$-uniform hypergraph on vertex set $V$ such that $H'_k \uplus H''_k \uplus D$ is $\l$-linear, $H''_k$ has $n_k-n'_k$ edges and, for each $x \in V$, $\deg_{H''_k}(x) \equiv \beta+\deg_D(x) \mod{k}$ and $\deg_{H''_k}(x) \leq 2k$.
Let
\[m_1=(n_k-n'_k)-\tfrac{1}{k}(\beta m +2d)\]
and note that $m_1$ is an integer since $2d\equiv \l m(m-1) \mod{k}$ by the definition of $d$ and $\beta \equiv -\l(m-1)  \mod{k}$. Furthermore, $|V(D)| < 2k < m_1 = O(1)$ using the definition of $n'_k$ and the fact that $0 \leq d \leq k-1$. Let $\{V_1,V_2\}$ be a partition of $V$ such that $|V_1|=m_1$, $V' \subseteq V_2$ and $V(D) \subseteq V_1$. Define
\[n''_{k}(x)=
\left\{
  \begin{array}{ll}
    \beta+k+\deg_D(x)\quad & \hbox{if $x \in V_1$} \\
    \beta & \hbox{if $x \in V_2$.}
  \end{array}
\right.\]
We obtain $H''_k$ by applying  Lemma~\ref{L:graphRealisation} with $\ell=k$, $r_x=n''_{k}(x)$ for each $x \in V$
and $A=\l K_{V'} \uplus D$. Now, $\sum_{x \in V}n''_{k}(x)=\beta m + k m_1+2d=k(n_k-n'_k)$ by the definition of $m_1$ and so (i) of Lemma~\ref{L:graphRealisation} holds with $n^\dag=n_k-n'_k$. Let $r$ and $q$ be as defined in the statement of Lemma~\ref{L:graphRealisation}. If $\beta \in \{1, \ldots,k-1\}$, then $r \leq \beta+k+1 \leq 2k$ and $q \leq |V'|+O(1) = \frac{1}{2}\gamma m +O(1)$. If $\beta =0$, then $r \leq k+1$ and $q \leq \frac{k^2}{\l}+1$, noting $n''_{k}(x)=0$ for each $x \in V'$. So, in either case, (ii) of Lemma~\ref{L:graphRealisation} can be seen to hold using $\gamma<\frac{1}{3k^2}$ and $n_k-n'_k \geq \tfrac{\beta m+d}{k}+2(k+1)(\tfrac{k^2}{\l}+1)$ by the definition of $n'_{k}$. Thus the desired hypergraph $H''_{k}$ exists, noting $H'_k \uplus H''_k \uplus D$ is $\l$-linear by our choice of $A$.\smallskip

\noindent \textbf{$\boldsymbol{H_{k+1}}$:} Let $G$ be the underlying multigraph of $H'_{k} \uplus H''_{k} \uplus D$ and $\overline{G}$ be its $\l$-defect. Then $H_{k+1}$ will be a $(k+1)$-uniform hypergraph with $n_{k+1}$ edges whose underlying multigraph is $\overline{G}$. Note that
\[\deg_{G}(x)=(k-1)(\deg_{H'_k \uplus H''_k}(x))+\deg_D(x)=
\left\{
  \begin{array}{ll}
    (k-1)(\beta+k)+k\,\deg_D(x) \ \ & \hbox{if $x \in V_1$} \\
    (k-1)(\deg_{H'_k}(x)+\beta) \ \ & \hbox{if $x \in V_2$}
  \end{array}
\right.
\]
and that $|N_G(x)| \leq |V'|-1+\beta = \frac{1}{2}\gamma m +O(1)$ for each $x \in V$. Thus, for each $x \in V$, $\deg_G(x) \equiv -\beta \mod{k}$ using the fact that $\deg_{H'_k}(x)\equiv 0 \mod{k}$. It follows that, for each $x \in V$, $|N^\l_{\overline{G}}(x)| > (1-\gamma)m$ and $\deg_{\overline{G}}(x) \equiv \l(m-1)+\beta \equiv 0 \mod{k}$. Finally, $|E(\overline{G})|=\l\binom{m}{2}-\binom{k}{2}n_{k}-d=\binom{k+1}{2}n_{k+1}$ by \eqref{E:edgeSumEasy}. Thus, by Lemma~\ref{T:denseFDecompsArbLambda} and our choice of $\gamma$, $H_{k+1}$ exists and will have $n_{k+1}$ edges.\smallskip

Then $H'_k \uplus H''_k \uplus H_{k+1}$ is a $\l$-linear hypergraph on $m$ vertices with $n_k$ edges of size $k$ and $n_{k+1}$ edges of size $k+1$ and the proof is complete by Lemma~\ref{L:zBoundFloor}.\smallskip

\textbf{Case 2.} This case proceeds analogously to Case 1. By the bounds on $n$ in this case and \eqref{E:nkExact} we have
\begin{equation} \label{E:bigEdgesLB}
\tfrac{\alpha m}{k+1}+2(k-1)\bigl(\tfrac{k(k-1)}{\l}+1\bigr)-\tfrac{d}{k} \leq n_{k+1} \leq \tfrac{\epsilon}{k}\tbinom{m}{2}.
\end{equation}
We will form the desired hypergraph $H$ as the union of $(k+1)$-uniform hypergraphs $H'_{k+1}$ and $H''_{k+1}$ and a $k$-uniform hypergraph $H_{k}$. We describe each of these hypergraphs and justify that they do indeed exist.\smallskip

\noindent \textbf{$\boldsymbol{H'_{k+1}}$:} This will be a (possibly empty) $(k+1)$-uniform $\l$-linear hypergraph on vertex set $V'$ such that $\deg_{H'_{k+1}}(x) \equiv 0 \mod{k-1}$ for each $x \in V'$ and $H'_{k+1}$ has $n'_{k+1}$ edges, where $n'_{k+1}$ is the greatest integer such that  $n'_{k+1} \equiv 0 \mod{\binom{h}{2}/\binom{k+1}{2}}$ and
\[n_{k+1}-n'_{k+1} \geq \tfrac{\alpha m}{k+1}+2(k-1)\bigl(\tfrac{k(k-1)}{\l}+1\bigr)-\tfrac{d}{k}.\]
We form $H'_{k+1}$ by decomposing each of $n'_{k+1}\binom{k+1}{2}/\binom{h}{2}$ edges of $H_h$ into $\binom{h}{2}/\binom{k+1}{2}$ edges of size $k+1$ and deleting the remaining edges of $H_h$. There are sufficiently
many edges in $H_h$ to accomplish this because $\binom{k+1}{2}n_{k+1} \leq \binom{|V'|}{2}$ by \eqref{E:bigEdgesLB} and since $|V'| \geq \frac{1}{2} \gamma m$ and $\epsilon = \frac{1}{3(k+1)}\gamma^2$. For each $x \in V$, we have $\deg_{H'_{k+1}}(x) \equiv 0 \mod{k-1}$ because $\deg_{H'_{k+1}}(x)$ is a multiple of $\frac{h-1}{k}$ and $h \equiv 1 \mod{k(k-1)}$. \smallskip

\noindent \textbf{$\boldsymbol{H''_{k+1}}$:} This will be a $(k+1)$-uniform hypergraph on vertex set $V$ such that $H'_{k+1} \uplus H''_{k+1} \uplus D$ is $\l$-linear, $H''_{k+1}$ has $n_{k+1}-n'_{k+1}$ edges and, for each $x \in V$, $\deg_{H''_{k+1}}(x) \equiv \alpha-\deg_D(x) \mod{k-1}$ and $\deg_{H''_{k+1}}(x) < 2k$.
Let
\[m_1=\mfrac{(k+1)(n_{k+1}-n'_{k+1})-\alpha m+2d}{k-1}\]
and note that $m_1$ is an integer since $\alpha \equiv \l(m-1) \mod{k-1}$, $2d \equiv \l m(m-1) - 2kn_{k+1} \mod{k-1}$ by \eqref{E:nkExact} and $n'_{k+1} \equiv 0 \mod{k-1}$ by the definitions of $n'_{k+1}$ and $h$. Furthermore, $|V(D)| < 2k < m_1 = O(1)$ by the definition of $n'_{k+1}$ and the fact that $0 \leq d \leq k-1$. Let $\{V_1,V_2\}$ be a partition of $V$ such that $|V_1|=m_1$, $V' \subseteq V_2$ and $V(D) \subseteq V_1$.
Define
\[n''_{k+1}(x)=
\left\{
  \begin{array}{ll}
    \alpha+k-1-\deg_D(x)\quad & \hbox{if $x \in V_1$} \\
    \alpha & \hbox{if $x \in V_2$.}
  \end{array}
\right.\]
We obtain $H''_{k+1}$ by applying  Lemma~\ref{L:graphRealisation} with $\ell=k+1$, $r_x=n''_{k+1}(x)$ for each $x \in V$ and $A=\l K_{V'} \uplus D$. Now, $\sum_{x \in V}n''_{k+1}(x)=\alpha m + (k-1) m_1-2d=(k+1)(n_{k+1}-n'_{k+1})$ by the definition of $m_1$ and so (i) of Lemma~\ref{L:graphRealisation} holds with $n^\dag=n_{k+1}-n'_{k+1}$. Let $r$ and $q$ be as defined in the statement of Lemma~\ref{L:graphRealisation}. If $\alpha \in \{1,\ldots,k-2\}$, then $r \leq \alpha+k-1 <2k$ and $q \leq |V'|+O(1) = \frac{1}{2}\gamma m +O(1)$. If $\alpha =0$, then $r \leq k-1$ and $q \leq \lfloor\frac{k(k-1)}{\l}+1\rfloor$. So, in either case, (ii) of Lemma~\ref{L:graphRealisation} can be seen to hold using $n_{k+1}-n'_{k+1} \geq \tfrac{\alpha m}{k+1}+2(k-1)(\frac{k(k-1)}{\l}+1)-\frac{d}{k}$ by \eqref{E:bigEdgesLB} and $\gamma<\frac{1}{3k^2}$. Thus the desired hypergraph $H''_{k+1}$ exists, noting $H'_{k+1} \uplus H''_{k+1} \uplus D$ is $\l$-linear by our choice of $A$.\smallskip

\noindent \textbf{$\boldsymbol{H_{k}}$:} Let $G$ be the underlying multigraph of $H'_{k+1} \uplus H''_{k+1} \uplus D$ and $\overline{G}$ be its $\l$-defect. Then $H_{k}$ will be a $k$-uniform hypergraph with $n_{k}$ edges whose underlying multigraph is $\overline{G}$. Note that
\[\deg_{G}(x)=k(\deg_{H'_{k+1} \uplus H''_{k+1}}(x))+\deg_D(x)=
\left\{
  \begin{array}{ll}
    k(\alpha+k-1)-(k-1)\deg_D(x) \ \ & \hbox{if $x \in V_1$} \\
    k(\deg_{H'_{k+1}}(x)+\alpha) \ \ & \hbox{if $x \in V_2$}
  \end{array}
\right.
\]
and that $|N_G(x)| \leq |V'|+O(1) = \frac{1}{2}\gamma m +O(1)$ for each $x \in V$. Thus, for each $x \in V$, $\deg_G(x) \equiv \alpha \mod{k-1}$ using the fact that $\deg_{H'_{k+1}}(x)\equiv 0 \mod{k-1}$. It follows that, for each $x \in V$, $|N^\l_{\overline{G}}(x)| > (1-\gamma)m$ and $\deg_{\overline{G}}(x) \equiv \l(m-1)-\alpha \equiv 0 \mod{k-1}$. Finally, $|E(\overline{G})|=\l\binom{m}{2}-\binom{k+1}{2}n_{k+1}-d=\binom{k}{2}n_{k}$ by \eqref{E:edgeSumEasy}. Thus, by Lemma~\ref{T:denseFDecompsArbLambda} and our choice of $\gamma$, $H_{k}$ exists and will have $n_{k}$ edges.\smallskip

Then $H'_{k+1} \uplus H''_{k+1} \uplus H_{k}$ is a $\l$-linear hypergraph on $m$ vertices with $n_k$ edges of size $k$ and $n_{k+1}$ edges of size $k+1$ and the proof is complete by Lemma~\ref{L:zBoundFloor}.
\end{proof}

\section{Achieving \texorpdfstring{$\boldsymbol{\lfloor B^k_\l(m,n) \rfloor}$}{B(m,n)}} \label{S:achievingB}

In this section we show that $Z_{2,\l+1}(m,n)=\lfloor B^k_\l(m,n) \rfloor$ in the cases that Theorem~\ref{T:boundAchieved} specifies and hence complete the proof of the theorem. As suggested by Lemmas~\ref{L:obeyLP} and \ref{L:bBoundWithEq}, we do so by constructing $\l$-linear hypergraphs containing edges of three different sizes.

In Section~\ref{S:achievingA}, it was a simple process to translate from an optimum of $\A_\l(m,n)$  with two nonzero real variables given by Lemma~\ref{L:aBoundWithEq} to integer numbers of edges of the two different sizes that we then showed could be realised in a $\l$-linear hypergraph achieving the total degree bound. It is not so obvious, however, how this should be done from an optimum of $\B_\l(m,n)$ with three nonzero variables given by Lemma~\ref{L:bBoundWithEq}. In the following lemma we accomplish this translation through a somewhat involved process that also produces a ``skeleton'' that prescribes exactly how many edges of each size will be incident with each vertex in the hypergraph. Once we have this skeleton, similar techniques to those we employed in the proof of Lemma~\ref{L:aBoundAchieved} can be used to construct the desired hypergraph.

\begin{lemma}\label{L:bBoundAchieved}
Let $k \geq 3$ be a fixed integer. For each sufficiently large integer $m$, we have $Z_{2,\l+1}(m,n)=\lfloor B^k_\l(m,n) \rfloor$ for all positive integers $\l$ and $n$ such that $\l(m-1) \not\equiv 0 \mod{k-1}$ and
\[\left(\l\tbinom{m}{2}-\tfrac{\alpha k}{k+1} m\right)/\tbinom{k}{2}+4(k-2)(k-1) \leq n \leq \left(\l\tbinom{m}{2}+m(k-1-\alpha)\right)/\tbinom{k}{2}-4(k+1)(k+2).\]
\end{lemma}

\begin{proof}
Because $\l(m-1) \not\equiv 0 \mod{k-1}$, we have $\alpha \in \{1,\ldots,k-2\}$. We will use this fact frequently throughout the proof. By Theorem~\ref{T:upperBounds} we have $Z_{2,\l+1}(m,n) \leq \lfloor B^k_\l(m,n) \rfloor$, so it only remains to find a $\l$-linear hypergraph $H$ on $m$ vertices with $n$ edges and total degree $\lfloor B^k_\l(m,n) \rfloor$.

Take $n^*_{k-1}$ and $n^*_{k+1}$ to be as defined in Lemma~\ref{L:bBoundWithEq}, that is,
\[n^*_{k-1}=\tfrac{(k+1)(k-1-\alpha)}{(k-1)(k^2-1-\alpha)}\bigl(n\tbinom{k}{2}-\l\tbinom{m}{2}+\tfrac{\alpha k}{k+1} m\bigr),
\quad
n^*_{k+1}= \tfrac{\alpha}{k^2-1-\alpha}\bigl(\l\tbinom{m}{2}-n\tbinom{k}{2}+m(k-1-\alpha)\bigr).\]
Note that $n^*_{k-1}$ and $n^*_{k+1}$ are positive by our assumed bounds on $n$. Our first goal is to decide, for each $\ell \in \{k-1,k,k+1\}$, the number $n_\ell$ of edges of size $\ell$ that $H$ will have and to define a skeleton that prescribes exactly how many edges of each size will be incident with each vertex in $H$. We are guided by the (not necessarily integral) values of $n^*_{k-1}$ and $n^*_{k+1}$. In order to accomplish this we make a sequence of definitions. With the exception of $m_1^*$, $\tau$ and $\nu$, all the quantities we define will be integers by their definitions.
\begin{itemize}
    \item
Let $\tau=B^k_\l(m,n)-\lfloor B^k_\l(m,n) \rfloor$ and let
\[m^*_1 = \tfrac{k-1}{k-1-\alpha}n^*_{k-1} = \tfrac{k+1}{k^2-1-\alpha}\left(n\tbinom{k}{2}-\l\tbinom{m}{2}+\tfrac{\alpha k }{k+1} m \right) \quad \text{and} \quad m_1=\lfloor m^*_1 \rfloor.\]
Observe that $B^k_\l(m,n)=\frac{\l(m-1)-\alpha}{k-1}m+m^*_1$ and $\lfloor B^k_\l(m,n) \rfloor=\frac{\l(m-1)-\alpha}{k-1}m+m_1$ and hence that $m^*_1-m_1=\tau$.
    \item
Let $n_{k-1}=\lceil\frac{k-1-\alpha}{k-1}m_1\rceil$ and note that $n_{k-1}=\lceil \frac{k-1-\alpha}{k-1}(m^*_1-\tau) \rceil=\lceil n^*_{k-1}-\frac{k-1-\alpha}{k-1}\tau \rceil$. Since $0 \leq \tau < 1$ we thus have $n_{k-1} \in \{\lfloor n^*_{k-1} \rfloor, \lceil n^*_{k-1} \rceil\}$. Let $\nu=n^*_{k-1}-n_{k-1}$ and note that $-1 < \nu <1$.
    \item
Let $n_{k+1}=\lfloor B^k_\l(m,n) \rfloor-kn+n_{k-1}$. Note that $n_{k+1}=n^*_{k+1}-\nu-\tau$ because $\lfloor B^k_\l(m,n) \rfloor=B^k_\l(m,n)-\tau$, $n_{k-1}=n^*_{k-1}-\nu$, and $B^k_\l(m,n)=kn-n^*_{k-1}+n^*_{k+1}$ by Lemma~\ref{L:bBoundWithEq}.
    \item
Let $n_k=n-n_{k-1}-n_{k+1}$.
    \item
Let $d=\l\binom{m}{2}-\sum_{\ell=k-1}^{k+1}\binom{\ell}{2}n_\ell=\l\binom{m}{2}-n\binom{k}{2}+(k-1)n_{k-1}-kn_{k+1}$. Note that $d=k\tau+\nu$ because $n_{k-1}=n^*_{k-1}-\nu$, $n_{k+1}=n^*_{k+1}-\nu-\tau$ and $n\binom{k}{2}-(k-1)n^*_{k-1}+kn^*_{k+1}=\l\binom{m}{2}$ (by calculation or by considering the proof of Lemma~\ref{L:bBoundWithEq}). Thus $d \in \{0,\ldots,k\}$, because $0 \leq \tau < 1$, $-1<\nu<1$ and $d$ is an integer.
    \item
Let $m_2=(k-1)n_{k-1}-m_1(k-1-\alpha)$. So $m_2=(k-1-\alpha)\tau-(k-1)\nu$ using $n_{k-1}=n^*_{k-1}-\nu$, $m_1=m^*_1-\tau$ and $n^*_{k-1}=\frac{k-1-\alpha}{k-1}m^*_1$. Observe that $0 \leq m_2 \leq k-2$ by the definition of $n_{k-1}$ and since $m_2$ is an integer.
\end{itemize}

With these definitions in hand we now define our skeleton as follows. Let $\{V_1,V_2,V_3,V_4\}$ be a partition of $V$ such that
$|V_1|=m_1$, $|V_2|=m_2$, $|V_3|=2d$ and $|V_4|=m-m_1-m_2-2d$. Note that $m-m_1-m_2-2d$ is nonnegative because
\[m_1 \leq m^*_1 \leq m-\tfrac{2k(k+1)(k+2)(k^2-1)}{k^2-1-\alpha} < m-2k(k+1)(k+2) < m-m_2-2d\]
where the second inequality follows using the definition of $m^*_1$ and the upper bound on $n$ in the hypothesis of this lemma, and the last follows because $m_2 \leq k-2$ and $d \leq k$.
Let $D$ be a matching with vertex set $V_3$. We will form the desired $\l$-linear hypergraph $H$ as the union of hypergraphs $H_{k-1}$, $H_k$ and $H_{k+1}$, where, for each $\ell \in \{k-1,k,k+1\}$, $H_\ell$ is $\ell$-uniform and has vertex set $V$ and $n_\ell$ edges. The $\l$-defect of $H$ will be $D$ and we will have $\deg_{H_\ell}(x)=n_\ell(x)$ for each $\ell \in \{k-1,k+1\}$ and $x \in V$, where
\[
(n_{k-1}(x),n_{k+1}(x))=
\left\{
  \begin{array}{ll}
    (k-1-\alpha,0) \quad & \hbox{if $x \in V_1$} \\
    (1,\alpha+1) \quad & \hbox{if $x \in V_2$} \\
    (0,\alpha-1) \quad & \hbox{if $x \in V_3$} \\
    (0,\alpha) \quad & \hbox{if $x \in V_4$.}
  \end{array}
\right.
\]

The total degree of such a hypergraph $H$ will be $(k-1)n_{k-1}+kn_k+(k+1)n_{k+1}=\lfloor B^k_\l(m,n) \rfloor$ by our definitions of $n_k$ and $n_{k+1}$. So it suffices to show that there exist hypergraphs $H_{k-1}$, $H_k$ and $H_{k+1}$ with the appropriate properties.

\smallskip \noindent $\boldsymbol{H_{k-1}.}$ We form $H_{k-1}$ by applying Lemma~\ref{L:graphRealisation} with $\ell=k-1$, $A=\l K_{V_2}$ and $r_x=n_{k-1}(x)$ for each $x \in V$. Now, $\sum_{x \in V}n_{k-1}(x)=m_1(k-1-\alpha)+m_2=(k-1)n_{k-1}$ by the definition of $m_2$ and so (i) of Lemma~\ref{L:graphRealisation} holds with $n^\dag=n_{k-1}$. Let $r$ and $q$ be as defined in the statement of Lemma~\ref{L:graphRealisation}. We have $r=k-1-\alpha$ and $q \leq \max\{1+\frac{(k-1-\alpha)(k-2)}{\l},|V_2|+\frac{k-2}{\l}\} \leq (k-1)(k-2)$, noting that $|V_2|=m_2 \leq k-2$. Then
\[n_{k-1} = \left\lceil \tfrac{k-1-\alpha}{k-1}(m^*_1-\tau) \right\rceil \geq  \tfrac{4(k-2)(k+1)(k-1-\alpha)}{(k^2-1-\alpha)}\tbinom{k}{2}-\tfrac{k-1-\alpha}{k-1}\tau >  2k(k-2)(k-1-\alpha)-1 > 2qr\]
where the first inequality follows using the definition of $m^*_1$ and the lower bound on $n$ in the hypothesis of this lemma. So (ii) of Lemma~\ref{L:graphRealisation} holds and the desired hypergraph $H_{k-1}$ exists. Further, $H_{k-1} \uplus \l K_{V_2}$ is $\l$-linear by our choice of $A$.

\smallskip \noindent $\boldsymbol{H_{k+1}.}$ We form $H_{k+1}$ by applying Lemma~\ref{L:graphRealisation} with $\ell=k+1$, $A=D$ and $r_x=n_{k+1}(x)$ for each $x \in V$. Now,
\begin{multline*}
\medop\sum_{x \in V}n_{k+1}(x)
=m_2(\alpha+1)+2d(\alpha-1)+\alpha(m-m_1-m_2-2d)=\alpha(m-m_1)+m_2-2d \\
= \alpha(m-m^*_1)-(k+1)(\nu+\tau)
= (k+1)(n^*_{k+1}-\nu-\tau)
= (k+1)n_{k+1}
\end{multline*}
where the third equality is obtained by substituting $m_1=m_1^*-\tau$, $m_2=(k-1-\alpha)\tau-(k-1)\nu$ and $d=k\tau+\nu$, and the fourth is obtained because $\alpha(m-m^*_1)=(k+1)n^*_{k+1}$ by calculation using the definitions of $m^*_1$ and $n^*_{k+1}$. Thus (i) of Lemma~\ref{L:graphRealisation} holds with $n^\dag=n_{k+1}$. Let $r$ and $q$ be as defined in the statement of Lemma~\ref{L:graphRealisation}. We have $r \leq \alpha+1$ and $q \leq \frac{k(\alpha+1)}{\l}+1$. Then
\[n_{k+1} = n^*_{k+1} - \nu -\tau \geq  \tfrac{4\alpha (k+1)(k+2)}{(k^2-1-\alpha)}\tbinom{k}{2}-\nu-\tau > 2\alpha k(k+2) > 2qr\]
where the first inequality follows using the definition of $n^*_{k+1}$ and the upper bound on $n$ in the hypothesis of this lemma and the last follows using $1 \leq \alpha \leq k-2$. So (ii) of Lemma~\ref{L:graphRealisation} holds and the desired hypergraph $H_{k+1}$ exists. Further, $H_{k-1} \uplus H_{k+1} \uplus D$ is $\l$-linear by our choice of $A$, because $H_{k-1} \uplus \l K_{V_2}$ is $\l$-linear, and because $n_{k-1}(x)=0$ for each $x \in V_3 \cup V_4$ and $n_{k+1}(x)=0$ for each $x \in V_1$.

\smallskip \noindent $\boldsymbol{H_{k}.}$ Let $G$ be the underlying multigraph of $H_{k-1} \uplus H_{k+1} \uplus D$ and $\overline{G}$ be its $\l$-defect.
Now
\[\deg_{G}(x)=(k-2)n_{k-1}(x)+kn_{k+1}(x)+\deg_D(x)=
\left\{
  \begin{array}{ll}
    (k-2)(k-1-\alpha) \quad & \hbox{if $x \in V_1$} \\
    (k-2)+k(\alpha+1) \quad & \hbox{if $x \in V_2$} \\
    k(\alpha-1)+1 \quad & \hbox{if $x \in V_3$} \\
    k\alpha \quad & \hbox{if $x \in V_4$}
  \end{array}
\right.
\]
and hence $|N_G(x)| < k^2=O(1)$ for each $x \in V$. Thus $\deg_G(x) \equiv \alpha \mod{k-1}$ for each $x \in V$. It follows that, for each $x \in V$, $|N^\l_{\overline{G}}(x)| \geq m-O(1)$ and $\deg_{\overline{G}}(x) \equiv \l(m-1)-\alpha \equiv 0 \mod{k-1}$. Finally, $|E(\overline{G})|=\l\binom{m}{2}-\binom{k-1}{2}n_{k-1}-\binom{k+1}{2}n_{k+1}-d$ and, since $d=\l\binom{m}{2}-\sum_{\ell=k-1}^{k+1}\binom{\ell}{2}n_\ell$, we have $|E(\overline{G})|=\binom{k}{2}n_{k}$. Thus, by Lemma~\ref{T:denseFDecompsArbLambda}, $H_{k}$ exists and will have $n_{k}$ edges.\smallskip

So, from our previous discussion, $H=H_{k-1} \uplus H_k \uplus H_{k+1}$ is a $\l$-linear hypergraph on $m$ vertices with $n$ edges and total degree $\lfloor B^k_\l(m,n) \rfloor$. This completes the proof.
\end{proof}

We can now prove Theorem~\ref{T:boundAchieved} and Corollary~\ref{C:asymptotic}.\pagebreak

\begin{proof}[\textup{\textbf{Proof of Theorem~\ref{T:boundAchieved}.}}]
Let $k \geq 2$ be a fixed integer, let $\alpha$ and $\beta$ be as given in the statement of the theorem, and suppose that $m$ is large. We refer to the three intervals specified for $n$ in the statement as the lower, middle and upper intervals respectively. When $n$ is in the middle interval the result follows from Lemma~\ref{L:aBoundAchieved}. Now suppose that $n$ is in the upper interval and hence that $\l(m-1) \not\equiv 0 \mod{k-1}$ since otherwise the upper interval is empty. The result follows from Lemma~\ref{L:bBoundAchieved}, noting that $\l(m-1) \not\equiv 0 \mod{k-1}$ implies $k \geq 3$. Finally suppose that $n$ is in the lower interval and hence that $\l(m-1) \not\equiv 0 \mod{k}$ since otherwise the lower interval is empty. Let $\ell=k+1$, let $\alpha_\ell$ be the least nonnegative integer such that $\l(m-1)-\alpha_\ell \equiv 0 \mod{k}$, and note that $\alpha_\ell \neq 0$ since $\l(m-1) \not\equiv 0 \mod{k}$. Then, by Lemma~\ref{L:bBoundAchieved},  $Z_{2,\l+1}(m,n)=\lfloor B^\ell_\l(m,n) \rfloor$ for all positive integers $\l$ and $n$ such that
\[\l\tbinom{m}{2}/\tbinom{\ell}{2} \leq n \leq \left(\l\tbinom{m}{2}+m(\ell-1-\alpha_\ell)\right)/\tbinom{\ell}{2}-4(\ell+1)(\ell+2).\]
The result follows by first substituting $\ell=k+1$ and then observing that $k-\alpha_\ell=\beta$.
\end{proof}

\begin{proof}[\textup{\textbf{Proof of Corollary~\ref{C:asymptotic}.}}]
If $a > 1$, then $k=1$ and by the result of \cite{Cul}, using $n \sim a \l \binom{m}{2}$,
\[Z_{2,\l+1}(m,n)=\l\tbinom{m}{2}+n \sim \l(a+1)\tbinom{m}{2},\]
confirming the corollary. So we may assume that $a \leq 1$ and hence that $k \geq 2$ and $1/\binom{k+1}{2} < a \leq 1/\binom{k}{2}$ by the definition of $k$. If $1/\binom{k+1}{2} < a < 1/\binom{k}{2}$, then by Theorem~\ref{T:boundAchieved}, using $n \sim a \l \binom{m}{2}$,
\[Z_{2,\l+1}(m,n) \sim A^k_\l(m,n) \sim \tfrac{\l}{k}\left(a\tbinom{k+1}{2}+1\right)\tbinom{m}{2}.\]
So the corollary holds whenever $a \neq 1/\binom{k}{2}$. This implies that $\lim_{m \rightarrow \infty}Z_{2,\l+1}(m,n)/\binom{m}{2}$ approaches $\frac{2\l}{k-1}$ as $a$ approaches $1/\binom{k}{2}$ (remember that $k$ in the statement of the corollary must be replaced with $k-1$ in our present notation when $a$ approaches $1/\binom{k}{2}$ from above). Thus, because $Z_{2,\l+1}(m,n)$ is increasing in $n$, the corollary also holds when $a = 1/\binom{k}{2}$.
\end{proof}

\section{Failing to achieve the upper bounds}\label{S:failing}

In this section we prove Theorem~\ref{T:boundNotAchieved}. Recall that we have defined, for positive integers $k \geq 2$ and $\l$, a \emph{symmetric design with block size $k$ and index $\l$} to be a $k$-uniform hypergraph whose underlying multigraph is $\l K_h$, where $h=\frac{k(k-1)}{\l}+1$.

The hypothesis $\l\binom{m}{2}-n\binom{k}{2} \equiv 0 \mod{k}$ and the bounds on $n$ in Theorem~\ref{T:boundNotAchieved}(i) mean that any $\l$-linear hypergraph on $m$ vertices with $n$ edges and total degree $\lfloor A^k_\l(m,n) \rfloor$ would have to have a very small number of edges of size $k+1$ and all remaining edges of size $k$. When combined with the condition that $\l(m-1) \equiv 0 \mod{k-1}$ this implies that the edges of size $k+1$ would have to be arranged in a specific fashion that turns out to be impossible given that there are so few of them. The proof of parts (ii) and (iii) proceeds along similar lines to that of part (i) with the roles of edges of size $k$ and edges of size $k+1$ exchanged. In this case, when $n = \binom{m}{2}/\binom{k+1}{2}+\frac{2h}{k+1}$, the special arrangement of the edges of size $k$ that the hypotheses force is exactly a symmetric design with block size $k$ and index $\l$.

\begin{proof}[\textup{\textbf{Proof of Theorem~\ref{T:boundNotAchieved}.}}]
Let $k$, $\l$, $m$ and $n$ be positive integers such that $k \geq 2$, $\l\binom{m}{2} - n\binom{k}{2} \equiv 0 \mod{k}$ and $\l\binom{m}{2}/\binom{k+1}{2} < n <  \l\binom{m}{2}/\binom{k}{2}$. Let $h=\frac{k(k-1)}{\l}+1$. Observe that $\l\binom{m}{2} - n\binom{k}{2} \equiv 0 \mod{k}$ implies that $\lfloor A^k_\l(m,n) \rfloor =A^k_\l(m,n)$. Thus, by Lemmas~\ref{L:obeyLP} and \ref{L:aBoundWithEq}, $Z_{2,\l+1}(m,n)=\lfloor A^k_\l(m,n) \rfloor$ if and only if there is a $\l$-linear hypergraph $H$ with $n_k$ edges of size $k$ and $n_{k+1}$ edges of size $k+1$, where $n_k=\tfrac{k+1}{2}n-\tfrac{\l}{k}\tbinom{m}{2}$ and $n_{k+1}=\tfrac{\l}{k}\tbinom{m}{2}-\tfrac{k-1}{2}n$. Note that both $n_k$ and $n_{k+1}$ are positive by our assumed bounds on $n$. Suppose such a hypergraph $H$ exists and note that $H$ must have empty $\l$-defect by \eqref{E:edgeSumEasy} since $\l\binom{m}{2} - n\binom{k}{2} \equiv 0 \mod{k}$. We will show this supposition leads to a contradiction to the hypothesised bounds on $n$ in the cases given in (i) and (ii) and that it implies the existence of the appropriate symmetric design in the case given by (iii). Let $V=V(H)$. Let $x$ be an arbitrary vertex in $V$ and, for each $\ell \in \{k,k+1\}$, let $n_\ell(x)$ be the number of edges of $H$ of size $\ell$ that contain $x$. \smallskip

\noindent\textbf{(i).} Suppose that $\l(m-1) \equiv 0 \mod{k-1}$. Thus, since $(k-1)n_k(x)+kn_{k+1}(x)=\l(m-1)$, we have $n_{k+1}(x) \equiv 0 \mod{k-1}$. In particular, if $n_{k+1}(x) \neq 0$ then $n_{k+1}(x) \geq k-1$. So, because $n_{k+1} > 0$, there is a vertex $y$ of $H$ that is in at least $k-1$ edges of size $k+1$. Each of these edges includes $k$ pairs of vertices containing $y$, and each pair of vertices occurs together in at most $\l$ edges. It follows that there are at least $\frac{k(k-1)}{\l}+1=h$ vertices of $H$ that are each in at least $k-1$ edges of size $k+1$. Thus $n_{k+1} = \frac{1}{k+1}\sum_{x \in V} n_{k+1}(x) \geq \frac{k-1}{k+1}h$. So we have shown that $n_{k+1} \geq \max\{\frac{k-1}{k+1}h,k-1\}$ and hence, because $n_{k+1}=\tfrac{\l}{k}\tbinom{m}{2}-\tfrac{k-1}{2}n$, we have $n \leq \l\binom{m}{2}/\binom{k}{2}-\frac{2}{k+1}\max\{h,k+1\}$. This proves (i). \smallskip

\noindent\textbf{(ii) and (iii).} Suppose that $\l(m-1) \equiv 0 \mod{k}$. Thus, since $(k-1)n_k(x)+kn_{k+1}(x)=\l(m-1)$, we have $n_k(x) \equiv 0 \mod{k}$. In particular, if $n_k(x) \neq 0$ then $n_k(x) \geq k$. So, because $n_k > 0$, there is a vertex $y$ of $H$ that is in at least $k$ edges of size $k$. Each of these edges includes $k-1$ pairs of vertices containing $y$, and each pair of vertices occurs together in at most $\l$ edges. It follows that there are at least $\frac{k(k-1)}{\l}+1=h$ vertices of $H$ that are each in at least $k$ edges of size $k$. Thus $n_{k} = \frac{1}{k}\sum_{x \in V} n_k(x) \geq h$. So we have shown that $n_{k} \geq \max\{h,k\}$ and hence, because $n_k=\tfrac{k+1}{2}n-\tfrac{\l}{k}\tbinom{m}{2}$, we have $n \geq \l\binom{m}{2}/\binom{k+1}{2} + \frac{2}{k+1}\max\{h,k\}$ which proves (ii). Now further suppose that $n = \l\binom{m}{2}/\binom{k+1}{2} + \frac{2h}{k+1}$. Then $n_k=\tfrac{k+1}{2}n-\tfrac{\l}{k}\tbinom{m}{2}=h$. Thus, because $h= n_{k} = \frac{1}{k}\sum_{x \in V} n_k(x)$ and there are at least $h$ vertices of $H$ that are each in at least $k$ edges of size $k$, it must be the case that exactly $h$ vertices of $H$ are in exactly $k$ edges of size $k$ and that each other vertex of $H$ is in no edge of size $k$. In this case simple counting shows that the hypergraph $H_k$ with vertex set $\{x \in V: n_k(x) > 0\}$ and whose edges are exactly the edges of size $k$ of $H$ must be a symmetric design with block size $k$ and index $\l$.

Now suppose that $n=\l\binom{m}{2}/\binom{k+1}{2} + \frac{2h}{k+1}$ and a symmetric design with block size $k$ and index $\l$ exists. This symmetric design is a $k$-uniform hypergraph $H_k$ with $h=n_k$ edges whose underlying multigraph is $\l K_h$. Thus, to show that $Z_{2,\l+1}(m,n)= \lfloor A^k_\l(m,n) \rfloor$, it suffices to show there is a $(k+1)$-uniform hypergraph with $n_{k+1}$ edges whose underlying multigraph is obtained from $\l K_m$ by deleting the edges of a submultigraph isomorphic to $\l K_h$. Note that $\l\binom{m}{2}-\l\binom{h}{2}=n_{k+1}\binom{k+1}{2}$. Furthermore, we have $\l(m-1) \equiv 0 \mod{k}$ and hence also $\l(m-1)-\l(h-1) \equiv 0 \mod{k}$ using the definition of $h$. So, by Lemma~\ref{T:denseFDecompsArbLambda}, provided that $m$ is sufficiently large compared to $k$, the required $(k+1)$-uniform hypergraph does indeed exist and hence $Z_{2,\l+1}(m,n)= \lfloor A^k_\l(m,n) \rfloor$.
\end{proof}

\section{Conclusion}\label{S:conc}

We have concentrated here on investigating the value of $Z_{2,\l+1}(m,n)$ in cases where $m$ is large and $n=\Theta(\l m^2)$. However, the bounds of Theorem~\ref{T:upperBounds} apply, and may be useful, in a wider range of situations. For example, the new bounds improve the best known upper bounds for $Z_{2,2}(m,n)$ reported in \cite[Table 1]{DamHegSzo} in the following situations. In each case the upper bound given in \cite{DamHegSzo} is larger by 1 than the one given below.

\begin{table}[H]
\begin{center}
\begin{tabular}{c||c|c|c|c|c|c|c|c|c|c}
                            $m$ & 14 & 14 & 17 & 17 & 18 & 18 & 20 & 21 & 22 & 23  \\ \hline
                            $n$ & 29 & 30 & 24 & 25 & 24 & 25 & 31 & 31 & 25 & 25  \\ \hline
                            $k$ &  3 &  3 &  4 &  4 &  4 &  4 &  4 &  4 &  5 &  5  \\ \hline
$\lfloor B^k_1(m,n) \rfloor$    & 87 & 89 & 92 & 94 & 97 & 99 & 124& 129& 119& 124
\end{tabular}\caption{Situations in which Theorem~\ref{T:upperBounds} improves the bounds reported in \cite{DamHegSzo}.}
\end{center}
\end{table}

Theorem~\ref{T:boundNotAchieved}(iii) indicates that there are cases where determining $Z_{2,\l+1}(m,n)$ exactly is probably very difficult. However, there remains the possibility of determining $Z_{2,\l+1}(m,n)$ in some of the cases not covered by Theorem~\ref{T:boundAchieved} and providing strong lower bounds in the others. For some $k$, it may be possible to completely determine the values of $Z_{2,\l+1}(m,n)$ for all the values of $n$ near $\binom{m}{2}/\binom{k}{2}$. In particular, this seems achievable when $k=3$ and this is one subject of our current research.

\subsection*{Acknowledgements}

Some of this research was undertaken while the first author visited Monash University. This visit was supported by Henan Normal University. The second author was supported by the Australian Research Council (grants FT160100048 and DP220102212). He thanks Donald Knuth who in 2015 asked him a question about the value of $Z_{2,2}(m,\frac{1}{3}\binom{m}{2})$ for $m \equiv 0,4 \mod{6}$ that eventually led to this line of research (cf. \cite[Exercise 7.2.2.2--488]{Knu} and its solution). He also thanks Charles Colbourn and Stefan Glock for useful discussions.


\begin{thebibliography}{10}

\bibitem{AloRonSza}
N. Alon, L. R\'{o}nyai and T. Szab\'{o}, Norm-graphs: variations and applications,
\textit{J. Combin. Theory Ser. B} \textbf{76} (1999), 280--290.

\bibitem{AloMelMubVer}
N. Alon, K.E. Mellinger, D. Mubayi and J. Verstra\"{e}te, The de {B}ruijn-{E}rd\H{o}s theorem for hypergraphs,
\textit{Des. Codes Cryptogr.}
\textbf{65} (2012), 233--245.

\bibitem{CarYus}
Y.~Caro and R.~Yuster, List decomposition of graphs,
\textit{Discrete Math.}
\textbf{243} (2002), 67--77.

\bibitem{ColRiaWalRad}
A. Collins, A. Riasanovsky, J. Wallace and S. Radziszowski, Zarankiewicz Numbers and Bipartite Ramsey Numbers,
\textit{Journal of Algorithms and Computation} \textbf{47} (2016) 63--78.

\bibitem{Con}
D. Conlon, Some remarks on the Zarankiewicz problem,
\textit{Mathematical Proceedings of the Cambridge Philosophical Society} \textbf{173} (2021) 1--7 doi:10.1017/S0305004121000475.

\bibitem{Cul}
K. \v{C}ul\'{\i}k, Teilweise l\"{o}sung eines verallgemeinerten problems von {K}. {Z}arankiewicz (German),
\textit{Ann. Polon. Math.}
\textbf{3} (1956), 165--168.

\bibitem{DamHegSzo}
G. Dam\'{a}sdi, T. H\'{e}ger and T. Sz\H{o}nyi, The {Z}arankiewicz problem, cages, and geometries,
\textit{Ann. Univ. Sci. Budapest. E\"{o}tv\"{o}s Sect. Math.}
\textbf{56}, (2013), 3--37.

\bibitem{GloKuhLoMonOst}
    S. Glock, D. K\"{u}hn, A. Lo, R. Montgomery and D. Osthus,
    On the decomposition threshold of a given graph,
    {\it J. Combin. Theory Ser. B} {\bf 139} (2019), 47--127.

\bibitem{GodHenOel}
W. Goddard, M. Henning and O. Oellermann, Bipartite Ramsey numbers and Zarankiewicz numbers,
\textit{Discrete Math.}
\textbf{219} (2000), 85--95.

\bibitem{Guy}
R.K. Guy, A many-facetted problem of Zarankiewicz, in: The many facets of graph theory, Springer (1969), 129--148.

\bibitem{Hyl}
C. Hylt\'{e}n-Cavallius, On a combinatorical problem,
\textit{Colloq. Math.}
\textbf{6}, (1958), 59--65.

\bibitem{KanTok}
M. Kano and N. Tokushige, Binding numbers and $f$-factors of graphs,
\textit{J. Combin. Theory Ser. B}
\textbf{54} (1992), 213--221.

\bibitem{Kee}
P. Keevash, The existence of designs,
arXiv:1401.3665 (2019).

\bibitem{Knu}
D.E. Knuth, The Art of Computer Programming, Volume 4, Fascicle 6: Satisfiability, Addison-Wesley Professional (2015).

\bibitem{KovSosTur}
T. K\"{o}vari, V. S\'{o}s and P. Tur\'{a}n, On a problem of {K}. {Z}arankiewicz,
\textit{Colloq. Math.}
\textbf{3}, (1954), 50--57.

\bibitem{Mor}
M. M\"{o}rs, A new result on the problem of Zarankiewicz,
\textit{J. Combin. Theory Ser. A} \textbf{31} (1981), 126--130.

\bibitem{Rei}
I. Reiman, \"{U}ber ein {P}roblem von {K}. {Z}arankiewicz (German),
\textit{Acta Math. Acad. Sci. Hungar.}
\textbf{9} (1958), 269--273.

\bibitem{Rei2}
I. Reiman, Su una proprietà dei 2-disegni (Italian),
\textit{Rend. Mat.}
\textbf{1} (1968), 75--81.

\bibitem{Rom}
S.~Roman, A problem of Zarankiewicz,
\textit{J. Combinatorial Theory Ser. A}
\textbf{18} (1975), 187--198.


\bibitem{Wes}
D.B.~West, Introduction to Graph Theory (2nd Edition), Prentice-Hall (2001).

\bibitem{Zar}
K. Zarankiewicz, Problem p101,
\textit{Colloq. Math.}, \textbf{2} (1951), 301.


\end{thebibliography}
\end{document}